\documentclass[11pt]{article}

\usepackage{amsmath,amsfonts,amssymb,amsthm}
\usepackage{lineno}
\usepackage{authblk}

\usepackage{geometry}
\usepackage[utf8]{inputenc} 
\usepackage[T1]{fontenc}    
\usepackage{hyperref}       
\usepackage{url}            
\usepackage{booktabs}       
\usepackage{amsfonts}       
\usepackage{nicefrac}       
\usepackage{microtype}      

\usepackage{bbm}
\usepackage{verbatim}
\usepackage{xcolor}

\usepackage{algorithm}
\usepackage{algpseudocode}
\usepackage{mathtools}

\usepackage{graphicx}
\graphicspath{ {./figures/} }
\usepackage[export]{adjustbox}
\usepackage{float}

\newcommand{\R}{\mathbb R}

\theoremstyle{plain}
\newtheorem{theorem}{Theorem}

\theoremstyle{definition}
\newtheorem{definition}[theorem]{Definition}

\theoremstyle{plain}
\newtheorem{proposition}[theorem]{Proposition}

\theoremstyle{plain}

\theoremstyle{plain}
\newtheorem{lemma}[theorem]{Lemma}

\theoremstyle{plain}
\newtheorem{remark}[theorem]{Remark}

\theoremstyle{plain}
\newtheorem{problem}[theorem]{Problem}

\theoremstyle{plain}
\newtheorem{conjecture}[theorem]{Conjecture}

\theoremstyle{plain}
\newtheorem{question}[theorem]{Question}

\theoremstyle{plain}
\newtheorem{counterexample}[theorem]{Counterexample}

\begin{document}


\title{\LARGE\textbf{Convolutions of Totally Positive Distributions with applications to Kernel Density Estimation}}

\author[1]{Ali Zartash}
\author[2]{Elina Robeva}
\affil[1]{\textit{Massachusetts Institute of Technology;\,\, aliz@alum.mit.edu}}
\affil[2]{\textit{The University of British Columbia;\,\, erobeva@math.ubc.ca}}
\date{}
\maketitle

\begin{abstract}
\normalsize
In this work we study the estimation of the density of a {\em totally positive} random vector. Total positivity of the distribution of a random vector implies a strong form of positive dependence between its coordinates and, in particular, it implies { positive association}. Since estimating a totally positive density  is a non-parametric problem, we take on a (modified) kernel density estimation approach. Our main result is that the sum of scaled standard Gaussian bumps centered at a min-max closed set provably yields a totally positive distribution. Hence, our strategy for producing a totally positive estimator is to form the {\em min-max closure} of the set of samples, and output a sum of Gaussian bumps centered at the points in this set. We can frame this sum as a convolution between the uniform distribution on a min-max closed set and a scaled standard Gaussian. 
We further conjecture that convolving any totally positive density with a standard Gaussian remains totally positive.


\end{abstract}





\section{Introduction}\label{sec:intro}

In this paper we study density estimation under the assumption that the original density fulfills the condition of {\em total positivity}, a specific, strong form of positive dependence. A distribution defined by a density function $f$ over $\mathcal{X}=\prod_{i=1}^d\mathcal{X}_{i}$, where each set $\mathcal{X}_i$ is totally ordered, is \emph{multivariate totally positive of order 2} (MTP$_2$) if
$$
f(x)f(y)\quad\leq \quad f(x\wedge y)f(x\vee y)\qquad\mbox{for all }x,y\in \mathcal{X},
$$
where $x\wedge y$ and $x\vee y$ are the coordinate-wise minimum and maximum. We restrict our attention to densities on $\mathcal X = \mathbb R^d$. 
If $f$ is strictly positive, then $f$ is MTP$_2$ if and only if $f$ is \emph{log-supermodular}, i.e., $\log(f)$ is supermodular:\vspace{-0.2cm}

$$\log(f(x)) + \log(f(y))\leq \log(f(x\vee y)) + \log(f(x\wedge y)).$$

Let $X_1,\dots , X_n\in\mathbb{R}^d$ be independent and identically distributed (i.i.d.) samples from an MTP$_2$ distribution with density function $f_0$. 
Computing an estimate for $f_0$ is a non-parametric density estimation problem. Attractive methods for solving such problems
include
kernel density estimation (KDE), adaptive smoothing, neighbor-based techniques, and non-parametric shape constraint estimation. For details we refer to the
surveys \cite{izenman1991review,turlach1993bandwidth, scott2015multivariate, silverman2018density,wand1994kernel,chen2017tutorial,wasserman2016topological,GJ_review}
 and references therein. We restrict our attention to the kernel density estimation approach. The general kernel density estimate $\hat{f}(\mathbf{x})$ of the density at a point $\mathbf{x} \in \R^d$ is given by
 \[
    \hat{f}(\mathbf{x}) = \frac{1}{nh^d}\sum_{i = 1}^{n} K \left(\frac{\mathbf{x} - X_i}{h}\right)
 \]
 where $K$ is a \textit{kernel function} and $h$ is the \textit{smoothing parameter}.
 Motivations for this estimator can be found, for example, in \cite[pp. 138-142]{scott2015multivariate}.
 
 The standard kernel density estimator would, in general, not be MTP$_2$. Instead, we augment our sample set to produce a new estimator which {\em is} MTP$_2$. We exploit the assumption that the true density is MTP$_2$ by substituting the sample $X = \{X_1,\ldots, X_n\}$ with its {\em min-max closure} MM$(X)$. By definition of the MTP$_2$ property, we know that any two points $a$ and $b$ in the sample $X$, satisfy $f_0(a)f_0(b)\leq f_0(a\wedge b)f_0(a\vee b)$ under the true density $f_0$. Thus, intuitively, the points $a\wedge b$ and $a\vee b$ could serve as even better representatives for the unknown distribution $f_0$ than $a$ and $b$ themselves, motivating us to add them to our sample. Performing this process repeatedly, we arrive at the min-max closure MM$(X)$ of the original sample $X$. We then convolve the uniform distribution on MM$(X)$ with a scaled standard Gaussian kernel. In other words, our estimator is a sum of scaled standard Gaussian bumps at the points in the {\em min-max closure of the sample} as opposed to the points in the original sample itself, which is the case in standard kernel density estimation. It is in general {\em not true} that convolving two MTP$_2$ functions yields an MTP$_2$ function. However, our main result (cf. Theorem~\ref{thm:main}) shows that the convolution is MTP$_2$ in the specific case of our estimator -- convolving (the indicator of) a min-max closed set with any scaled {\em standard} Gaussian kernel yields an MTP$_2$ density. In the process of proving this result we provide a general condition under which the convolution of an MTP$_2$ function and a scaled standard Gaussian is MTP$_2$ (cf. Theorem~\ref{thm:conv}) and we prove that the indicator function of a min-max closed set satisfies this condition (cf. Proposition~\ref{prop}). Finally, we conjecture that convolving any MTP$_2$ density with a scaled standard Gaussian remains MTP$_2$.

\subsection{Related work}\label{sec:work}

MTP$_2$ was introduced in~\cite{fortuin1971correlation} and it implies \emph{positive association}, an important property in probability theory and statistical physics, which is usually difficult to verify. In fact, most notions of positive dependence are implied by MTP$_2$; see for example~\cite{colangelo2005some, colangelo2006some} for an overview. Furthermore, the class of MTP$_2$ distributions has desirable properties \cite{muller2000some}, for example, it is closed under marginalization and conditioning~\cite{KR} (in particular, Yule-Simpson's Paradox cannot happen for an MTP$_2$ distribution), and independence models generated by MTP$_2$ distributions are compositional semigraphoids~\cite{MTP2Markov2015}. The special case of Gaussian MTP$_2$ distributions was studied by Karlin and Rinott~\cite{karlinGaussian} and also in~\cite{slawski2015estimation, LUZ} from 
the perspective of MLE and optimization.

Even though the MTP$_2$ condition is quite strong (recall that it implies all other forms of positive dependence), a variety of distributions are intrinsically MTP$_2$. For example, order statistics of i.i.d. variables~\cite{KR}, ferromagnetic Ising models~\cite{Leb72}, Brownian motion tree models, and Gaussian and binary latent tree models~\cite{Zwiernik}.

With regards to density estimation, the MTP$_2$ assumption is a shape constraint on $f_0$, the true density at hand. Estimation under MTP$_2$ has been studied in the recent past.
In~\cite{RobStuTraUhl18}, the authors study maximum likelihood estimation under MTP$_2$; since the likelihood function is unbounded over the set of MTP$_2$ distributions, they impose the additional restriction of log-concavity on the unknown density. In other work,~\cite{HutMaoRigRob19, HutMaoRigRob20}, the authors study the minimax rates of estimating 2-dimensional MTP$_2$ densities with bounded domain.

In the present paper, we take a modified kernel density estimation approach for estimating general MTP$_2$ densities.

\subsection{Organization}
The rest of the paper is organized as follows. In Section~\ref{sec:tpkde} we define and describe our estimator, the {\em Totally Positive Kernel Density Estimator (TPKDE)}, in detail. In Section~\ref{sec:conv} we discuss convolutions and compositions of totally positive functions in a general setting and introduce Theorem~\ref{thm:conv} and Proposition~\ref{prop}. In Theorem~\ref{thm:conv} we provide a sufficient condition for the convolution of two MTP$_2$ functions to be MTP$_2$. In Proposition~\ref{prop} we show that this condition is satisfied by our estimator (and actually any min-max closed set), meaning that it yields an MTP$_2$ density. This is a combinatorial result of independent mathematical interest, and its proof is located in Section~\ref{sec:prop_proof}. In Section~\ref{sec:conjecture} we discuss extending Theorem~\ref{thm:main} to {\em any} MTP$_2$ distribution. In Section~\ref{sec:algs} we present the algorithmic implementations we devised and used to compute the min-max closure of a set of points. In Section~\ref{sec:experiments} we provide experimental results comparing our estimator to the standard kernel density estimator. We conclude our study in Section~\ref{sec:conclusion}, where we raise a multitude of questions for future work.


\section{The Totally Positive Kernel Density Estimator (TPKDE)}\label{sec:tpkde}

Instead of using standard kernel density estimation, in which one sums up, say, Gaussian bumps centered at the points in our sample set $X = \{X_1,\ldots, X_n\}$, we propose to exploit the MTP$_2$ condition in order to obtain a better estimator. Recall that given two points $X_i$ and $X_j$, under the true density, $f_0$, we have 
$$f_0(X_i)f_0(X_j)\leq f_0(X_i\wedge X_j)f_0(X_i\vee X_j).$$ Therefore, intuitively the (coordinate-wise) minimum $X_i\wedge X_j$ and the maximum $X_i\vee X_j$ have a higher likelihood of occurring than $X_i$ and $X_j$ themselves. This is why we add $X_i\wedge X_j$ and $X_i\vee X_j$ to $X$. Continuing this process recursively, we arrive at the {\em min-max closure} of $X$, denoted MM$(X)$. 

\begin{definition} The {\em min-max closure} of a finite set of points $X = \{X_1,\ldots, X_n\}\subset\mathbb R^d$ is the smallest set $S$ such that $X\subseteq S$, and $S$ is {\em min-max closed}, i.e., for every $a,b\in S$, their minimum $a\wedge b$ and maximum $a\vee b$ are also in $S$.
\end{definition}

Note that the process described above terminates since MM$(X)$ is contained in the finite set of points $Z$ such that for each $i$, the $i$-th coordinate of $Z$ equals the $i$-th coordinate of one of $X_1,\ldots, X_n$.
We sometimes call this finite set the {\em grid generated by $X_1,\ldots, X_n$}.

In our setting it is assumed that the true density is MTP$_2$. Thus, intuitively, we can see that using MM$(X)$ instead of $X$ would give us a ``more representative'' set of points than $X$ and, therefore, would yield a better estimator.

Accordingly, in our modified version of the kernel density estimator, we replace the role of $X$ by its min-max closure MM$(X)$. Without this modification, we would in general {\em not} get an MTP$_2$ density. We can see why this is the case by a simple counterexample.
 \begin{counterexample}
Consider the average of two 2-dimensional standard Gaussians centered at $(0,1)$ and $(1,0)$ (i.e. $X = \{(1,0),(0,1)\}$ is our sample) and let us call this density $\hat{f}$. To see that it is not MTP$_2$, consider the points $p_1 = (0,1)$ and $p_2 = (1,0)$. Then $p_1 \wedge p_2 = (0,0)$, $p_1 \vee p_2 = (1,1)$,  $\hat{f}(p_1)\hat{f}(p_2) = \hat{f}(0,1)\hat{f}(1,0) \approx 0.012$, and $\hat{f}(p_1 \wedge p_2)\hat{f}(p_1 \vee p_2) = \hat{f}(0,0)\hat{f}(1,1) \approx 0.009$. So
$$
\hat{f}(p_1)\hat{f}(p_2) > \hat{f}(p_1 \wedge p_2)\hat{f}(p_1 \vee p_2)
$$
which shows that the MTP$_2$ condition is not satisfied.

\end{counterexample}
We wish to remark that we need to use the multivariate scaled {\em standard} Gaussian kernel in $d$-coordinates for our kernel (see Section~\ref{sec:conv} for why we need this kernel as a opposed to, for example, a general (MTP$_2$) Gaussian kernel). We now define our estimator.

\begin{definition}\label{def:TPKDE}[Totally Positive Kernel Density Estimator]
Given $n$ i.i.d. samples in the sample set $X = \{X_1,\ldots, X_n\}$, each $X_i \in \R^d$ drawn from a true density $f_0$, denote the min-max closure by MM$(X)$ and the size of MM$(X)$ by $m$. We can then write MM$(X)$ as MM$(X) = \{X_1,...X_n,...,X_m\}$, where $m-n \geq 0$ is the number of additional points we added to $X$ to transform it into MM$(X)$. The {\em Totally Positive Kernel Density Estimator} (or {\em TPKDE}) $\hat{f}(\mathbf{x})$ at a point $\mathbf{x} \in \R^d$ is then
$$
\hat{f}(\mathbf{x}) = \frac{1}{m} \sum_{i = 1}^{m} \frac{1}{h^d} N_d\left(\frac{\mathbf{x} - X_i}{h}\right)
$$
where $h \in \R, h > 0$ and $N_d$ is the standard Gaussian kernel (with identity covariance matrix) in $d$-coordinates.
Note that we can also write the {\em TPKDE} equivalently as
\begin{align}\label{eqn:TPKDE}
\hat{f}(\mathbf{x}) = \frac{1}{m} \sum_{i = 1}^{m} N_d^h\left(\mathbf{x} - X_i \right)
\end{align}
where $N_d^h$ is a multivariate Gaussian density in $d$-coordinates with covariance matrix, $\Sigma = h^2I$, where $I$ is the $d\times d$ identity matrix.
\end{definition}
We saw above that the unmodified kernel density estimator need not be MTP$_2$. We show that the TPKDE is, in fact, MTP$_2$ as stated in the following theorem. 

\begin{theorem}\label{thm:main}
The totally positive kernel density estimator is an MTP$_2$ density.
\end{theorem}

This result is a consequence of Theorem~\ref{thm:conv} and Proposition~\ref{prop}, and its proof is stated at the end of Section~\ref{sec:conv}. The result is interesting not only because of its use for kernel density estimation, but also because typically convolutions of MTP$_2$ functions are not MTP$_2$; i.e., the set of MTP$_2$ functions is \textit{not closed} under convolution. We devote Section~\ref{sec:conv} to a discussion of this matter.

Also note that for our estimator to be MTP$_2$ we need to use a scaled {\em standard} Gaussian kernel, i.e., with covariance matrix $\Sigma = hI$, a multiple of the identity matrix. In Section~\ref{sec:conv}, Counterexample~\ref{counter:2d_mtp2} shows that using an MTP$_2$ Gaussian kernel with different covariance could give rise to a non-MTP$_2$ density.

\section{Convolutions of MTP$_2$ Functions}\label{sec:conv}

The Totally Positive Kernel Density estimator for a sample set $X = \{X_1,\ldots, X_n\}$ with min-max closure MM$(X) = \{X_1,\dots,X_m\}$ given in~\eqref{eqn:TPKDE}
can actually be seen as a convolution of an MTP$_2$ distribution with a scaled standard Gaussian. For a function $g : \R^d \rightarrow \R$ and a distribution $Q$ on $\mathbb  R^d$, we define their {\em convolution} $C(\mathbf{a}) : \R^d \rightarrow \R$ at $(a_1,\dots,a_d)=\mathbf{a} \in \R^d$ as 
\begin{align}\label{eqn:conv}
C(\mathbf{a}) \overset{\Delta}{=} \int_{\mathbf{x} \in \R^d}  g(\mathbf{a}-\mathbf{x}) Q(d\mathbf{x}).
\end{align}
Here $Q$ is either a distribution with Lebesgue density $\alpha$ or a discrete distribution expressible as a sum of finitely many dirac-delta functions with weights $\alpha(x) = Q(\{x\})$.

 Note that the totally positive kernel density estimator is
\begin{align}\label{eqn:TPKDEconv}
\hat{f}(\mathbf{a}) = C(\mathbf{a})
\end{align}
with
$$
g = N_d^h,
$$
and $Q$ is the uniform distribution on MM$(X)$.
 
 This convolution is similar to other types used in many {\em composition theorems} of  MTP$_2$ kernels that have been proved and discussed in great detail by Karlin, for example, in \cite{karlin_1968}. However, the type of convolution we have defined above is generally omitted from any such previous discussion. This is because it, in general, does {\em not} satisfy the MTP$_2$ condition. Specifically, despite both of $g$ and $Q$ being MTP$_2$, it is generally not true that $C$ is MTP$_2$.
 
 One such counterexample is given by Karlin and Rinott in \cite[pp. 486]{KR}. They use the result from \cite{karlinGaussian} that a multivariate Gaussian density is MTP$_2$ if and only if the inverse of its covariance matrix has only non-positive off-diagonal entries (i.e. is an {\em M-matrix}). They provide two 3-dimensional positive definite matrices $A$ and $B$ whose inverses are  M-matrices, but the inverse $(A+B)^{-1}$ of $A+B$ is not an $M$-matrix. (Note that the convolution of the two Gaussians with covariances $A$ and $B$ is another Gaussian distribution with covariance $A+B$.) Therefore, the convolution of two MTP$_2$ Gaussians need not be MTP$_2$ and, accordingly, two general MTP$_2$ functions need not convolve to produce an MTP$_2$ function.
 
 We provide another counterexample in 2-dimensions, that is also relevant to our setting. We show that convolving the uniform distribution of a min-max closed set with a {\em non-standard} MTP$_2$ Gaussian need not be MTP$_2$.
 
 \begin{counterexample}\label{counter:2d_mtp2}
 Consider the set of points $X = \{(2,0),(0,1)\}$. Its min-max closure is MM$(X) = \{(2,0),(0,1),(2,1),(0,0)\}$. Then the unifrom distribution on this set $Q = U_{\text{MM}(X)}$ is obviously MTP$_2$ due to the definition of a min-max closed set. However, if we convolve $Q$ with an MTP$_2$ Gaussian $g$ whose inverse covariance matrix is the $M$-matrix
 $$
 \Sigma^{-1}= \begin{pmatrix} 5 & -2 \\
                             -2 & 1  
                             \end{pmatrix}
 $$
 the result is not MTP$_2$. For example, if $p_1 = (0.98, 0.43)$ and $p_2 = (0.49, 0.7 )$, and if $C$ denotes the convolution of $g$ and $Q$ we get that 
 $$
    C(p_1 \wedge p_2)C(p_1 \vee p_2) < C(p_1)C(p_2).
 $$
 This is a violation of the MTP$_2$ condition.
  \end{counterexample}
  
A more detailed discussion of this counterexample is provided in Appendix~\ref{appendix:B}.

 Thus, we know not all MTP$_2$ Gaussian densities can be used as the kernel in our estimator. We now show (cf. Theorem~\ref{thm:conv}) that choosing a covariance of $\Sigma = hI$ and imposing a further restriction on $Q$ (cf. Constraint A) yields an MTP$_2$ convolution. Then in Proposition~\ref{prop} we show that if $Q$ is the uniform distribution on a min-max closed set (which is what our estimator uses) it satisfies the sufficient Constraint A, and therefore, our proposed totally positive kernel density estimator from Definition~\ref{def:TPKDE} is an MTP$_2$ density.

\begin{theorem}\label{thm:conv}
	Let $N_d^h(\cdot) : \R^d \rightarrow \R $ be a $d$-dimensional multivariate Gaussian density function with covariance matrix $\Sigma = hI$, i.e., with independent coordinates and a common variance $h$. Let $Q$ be a distribution on $\mathbb R^d$ which either has density $\alpha$ on $\mathbb R^d$ or is discrete with weights $\alpha(x) = Q(\{x\})$. Suppose further that $\alpha$ satisfies the constraint
	\\
	\\
	\emph{\textbf{Constraint A. }\label{cA}
    Let $x_{10} < x_{11}, x_{20} < x_{21}, \ldots, x_{d0} < x_{d1}$, where $x_{ij}\in\mathbb R$ for all $1\leq i\leq d$ and $0\leq j\leq 1$. For each binary string $b \in \{0, 1 \}^d$, denote $x_b := (x_{1b_1}, x_{2b_2} \ldots, x_{db_d})\in\mathbb R^d$, and let $\overline b := (1-b_1,\ldots, 1-b_d)$ be the complementary binary string. Then for any permutation $\pi$ of $\{1,\ldots, d\}$ we have
	\begin{align}
	\sum_{a\in\{0,1\}^{d-2}} \left(\alpha(x_{\pi(11a)})\alpha(x_{\pi(00\overline a)}) - \alpha(x_{\pi(10a)}) \alpha(x_{\pi(01\overline a)})\right) \geq 0.
	\end{align}}
	\\
	Then the $d$-dimensional convolution of $N_d^h(\cdot)$ with $Q$ is an MTP$_2$ function.
	In other words, the function
	$$
		C(a_1, \dots ,a_d) = \int_{(x_1, \dots ,x_d) \in \R^d}^{}  N_d^h(a_1 - x_1, \dots , a_d - x_d)Q(
		dx_1\dots dx_d),
	$$
	is MTP$_2$.
\end{theorem}

\begin{remark}\label{remark:ConsA}
Note that Constraint A gives a restriction on the values of $\alpha$ at the vertices of the hypercube $\prod_{i=1}^d \{x_{i0}, x_{i1}\}$. For example, when $d=2$, it is {
\em equivalent} to the MTP$_2$ condition. When $d=3$, the constraint simply says
$$\alpha(x_{\pi(111)})\alpha(x_{\pi(000)}) - \alpha(x_{\pi(101)})\alpha(x_{\pi(010)})$$
$$ + \alpha(x_{\pi(110)})\alpha(x_{\pi(001)}) - \alpha(x_{\pi(100)})\alpha(x_{\pi(011)}) \geq 0,$$
a condition involving all of the vertices of the hypercube $\{x_{10}, x_{11}\}\times \{x_{20}, x_{21}\}\times \{x_{30}, x_{31}\}$ in $\mathbb R^3$.
\end{remark}

We provide the proof of this theorem (Theorem~\ref{thm:conv}) in Appendix~\ref{appendix:A}. Our next result (Proposition~\ref{prop}) is the final ingredient to showing Theorem~\ref{thm:main}. 
It is an interesting combinatorial result on its own and we provide its proof in Section~\ref{sec:prop_proof}.

\begin{proposition}\label{prop}
Let $Q$ be the uniform distribution over a min-maxed closed set $M \subseteq \R^d$ (which is potentially a finite set), then, $Q$ satisfies Constraint A.
\end{proposition}


Note that the set we use in the TPKDE is a min-max closed set (it is just the min-max closure of the set of sample points). Thus we can now state the simple proof that the TPKDE yields an MTP$_2$ density.

\begin{proof}[Proof of Theorem~\ref{thm:main}]
The TPKDE uses a min-max closed set MM$(X)$ where $X$ is the original sample set. Now, from Proposition~\ref{prop} we know that the uniform distribution $Q$ over MM$(X)$, satisfies Constraint A. Since the TPKDE can be framed as a convolution of $Q$ and a scaled standard Gaussian as shown in \eqref{eqn:TPKDEconv}, we can use Theorem~\ref{thm:conv} to immediately conclude that the totally positive kernel density estimator is MTP$_2$.
\end{proof}

\section{Proof of Proposition~\ref{prop}}\label{sec:prop_proof}
Proposition~\ref{prop} is much more general than its use here in the TPKDE setting. It states that if we take \emph{any} min-max closed set, then the uniform distribution over it (or the set's indicator function) will satisfy Constraint A, which is also a general constraint. This result is of independent mathematical interest as well.

We first introduce a lemma that will be used in the proof of Proposition~\ref{prop}.

\begin{lemma}\label{lem:bin_string}
For a binary string $b$ of arbitrary length $d$, i.e., $b \in \{0,1\}^d$, we use $\overline b := (1-b_1,\ldots, 1-b_d)$ to denote the complementary binary string. For any two binary strings, $a,b \in \{0,1\}^d$, we have that
$$
  \left( \overline a \wedge b \right)=  \overline {\left( a \vee \overline b \right)}.
$$
\end{lemma}

\begin{proof}
Let $\mathcal A \subseteq [d]$ be the set of indices where the string $a$ has 1s. And similarly, let $\mathcal B \subseteq [d]$ be the set of indices where the string $b$ has 1s. Then $\left( \overline a \wedge b \right)$ has 1s in the index set $(A^c)\cap \mathcal B$.
But if we look at $\left( a \vee \overline b \right)$, it has 1s in the set  $\mathcal A \cup \mathcal B^c$. So then its complement, $\overline {\left( a \vee \overline b \right)}$, would have 1s in the remaining indices i.e. $(\mathcal A \cup \mathcal B^c)^c = \mathcal A^c \cap \mathcal B$. But this is exactly the set of indices in which $\left( \overline a \wedge b \right)$ had 1s, thus proving this lemma.
\end{proof}

\noindent Now, we can start proving Proposition~\ref{prop}.
\begin{proof}[Proof of Proposition~\ref{prop}]
Let $x_{10} < x_{11}, x_{20} < x_{21}, \ldots, x_{d0} < x_{d1}$, where $x_{ij}\in\mathbb R$ for all $1\leq i\leq d$ and $0\leq j\leq 1$. For each binary string $b\in\{0,1\}^d$, denote $x_b := (x_{1b_1}, \ldots, x_{db_d})\in\mathbb R^d$, and let $\overline b := (1-b_1,\ldots, 1-b_d)$ be the complementary binary string. By symmetry, let $\pi$ be the identity permutation on $\{1,2,\ldots, d\}$. We will show that
$$\sum_{a\in\{0,1\}^{d-2}} \left(\alpha(x_{11a})\alpha(x_{00\overline a}) - \alpha(x_{10a}) \alpha(x_{01\overline a})\right) \geq 0.$$
where $\alpha(x) : \R^d \rightarrow  \{0, c\}$ equals $c>0$ on a min-maxed closed set $M \subseteq \R^d$ and 0 outside of it.

If $M$ is a finite set, then $\alpha(x) = c := 1/|M|$ when $x\in M$, and 0 when $x\not\in M$. Otherwise $\alpha$ is the density of the uniform probability measure on $M$, and $\alpha(x) = c$ (for the appropriate $c \in \R$) for $x \in M$ and 0 otherwise.

Let $\mathcal B_1 = \{a\in\{0,1\}^{d-2} : \alpha(x_{10a}) =c > 0\}$. Note that 
$x_a \wedge x_b = x_{a\wedge b}$ and $x_a \vee x_b = x_{a\vee b}$. Now since $M$ is min-max closed, so is $\mathcal B_1$. This is because if $a,b \in \mathcal B_1$, then $\alpha(x_{10a}) = \alpha(x_{10b}) = c$ which means $x_{10a},x_{10b} \in M$ so $x_{10a} \vee x_{10b} = x_{10a\vee b} \in M$ (due to its min-max closure). This means $\alpha(x_{10a\vee b}) = c$, hence $a \vee b \in \mathcal B_1$, and similarly for the minimum $a \wedge b$.

Let $\mathcal B_2 := \{{a}\in\{0,1\}^{d-2} : \alpha(x_{01 a}) > 0\}$. Similarly, since $M$ is min-max closed, so is $\mathcal B_2$. We use the notation $\overline{\mathcal B_2}$ to denote the set of complementary strings of $\mathcal B_2 $, and we can see that $\overline{\mathcal B_2}$ is also min-max closed. This is because when we take the complementary strings, coordinate-wise minima become coordinate-wise maxima and vice-versa, as shown in Lemma~\ref{lem:bin_string} - precisely,
$$
a,b \in \overline{\mathcal B_2} \implies \overline{a},\overline{b} \in \mathcal B_2 \implies \overline{a} \vee \overline{b} \in \mathcal B_2 \implies \overline{\overline{a} \vee \overline{b}} = a \wedge b \in \overline{\mathcal B_2},
$$
and similarly for $a \vee b$.

Now, let $\mathcal A = \mathcal B_1 \cap \overline{\mathcal B_2}$. This is precisely the set of $a\in\{0,1\}^{d-2}$ such that $\alpha(x_{10a}) = c$ and $\alpha(x_{01\overline a}) = c$, i.e. $\alpha(x_{10a})\alpha(x_{01\overline a}) = c^2$. Moreover, $\mathcal A$ is min-max closed (intersection of two min-max closed sets).

We now let $a^*$ be the minimum of $\mathcal A$ (i.e. the indices where $a^*$ has 1s are common to all $a \in \mathcal{A}$). Therefore, since $\overline{\mathcal A} = \overline{\mathcal B_1} \cap \mathcal B_2$ (i.e. the set of $\overline{a}\in\{0,1\}^{d-2}$ such that $\alpha(x_{10a})\alpha(x_{01\overline a}) > 0$), necessarily $\overline{a^*}$ is the maximum of $\overline{\mathcal A}$. To see this, let $\mathcal{O}(x)$ be the set of indices where the binary string $x$ has ones. Then
\begin{align*}
            &\forall b \in \mathcal A, \mathcal O(a^*) \subseteq \mathcal O(b)
           \ \ \ \ \ \ \ \text{(since }a^*\text{ is minimum of } \mathcal A \text{)}\\
    \implies &\forall b \in \mathcal A,\overline{\mathcal O(a^*)} \supseteq \overline{\mathcal O(b)}\\
    \implies &\forall b \in \mathcal A, {\mathcal O(\overline{a^*})} \supseteq {\mathcal O(\overline b)}\\
    \implies &\forall \overline{b} \in \overline{\mathcal A},
        {\mathcal O(\overline{a^*})} \supseteq {\mathcal O(\overline b)}
        \ \ \ \ \ \ \ \ \text{i.e. }\overline{a^*}\text{ is maxmium of }\overline{\mathcal A} \text{.}
\end{align*}

So now, for every $a\in\mathcal A$, we have that $\overline{a}\in \mathcal B_2$, thus, $\alpha(x_{01\overline a}) > 0$. Therefore, since $\mathcal M$ is min-max closed, $\alpha(x_{10a^*}\vee x_{01\overline a}) = c>0$ and $\alpha(x_{10a}\wedge x_{01\overline{a^*}}) = c >0$. But note that $x_{10a^*}\vee x_{01\overline a} = x_{11(a^*\vee \overline a)}$ and $x_{10\overline{a^*}}\wedge x_{01\overline a} = x_{00(\overline{a^*}\wedge a)}$. Moreover, by Lemma~\ref{lem:bin_string}, $\overline{\overline{a^*}\wedge a} = a^*\vee \overline a$. Finally, since $a^*$ is the minimum over all $a\in\mathcal A$, and $\overline{a^*}$ is the maximum over all $\overline a$ for $a\in\mathcal A$, then, for all $a\neq b \in \mathcal A$, $a^*\vee \overline a \neq a^*\vee \overline b$ (as in the indices where $a^*$ has 1s, $\overline{a}$ and $\overline{b}$ will both have 0s anyways), and similarly, $\overline{a^*}\wedge a\neq \overline{a^*}\wedge b$.

Now let $\mathcal C = \{a^*\vee \overline a: a\in\mathcal A\}$. Then, finally,
$$\sum_{a\in\{0,1\}^{d-2}} \left(\alpha(x_{11a})\alpha(x_{00\overline a}) - \alpha(x_{10a}) \alpha(x_{01\overline a})\right) $$
$$= \sum_{a\in\{0,1\}^{d-2}} \alpha(x_{11a})\alpha(x_{00\overline a}) -\sum_{a\in\{0,1\}^{d-2}} \alpha(x_{10a}) \alpha(x_{01\overline a}) $$
$$=\sum_{a\in\{0,1\}^{d-2}}\alpha(x_{11a})\alpha(x_{00\overline a}) - \sum_{a\in\mathcal A}\alpha(x_{10a}) \alpha(x_{01\overline a}) $$
$$= \sum_{a\in\{0,1\}^{d-2}\setminus \mathcal C} \alpha(x_{11a})\alpha(x_{00\overline a}) + \sum_{a\in\mathcal A}\left(\alpha(x_{11(a^*\vee \overline a)})\alpha(x_{00(\overline{a^*}\wedge a)})-\alpha(x_{10a}) \alpha(x_{01\overline a})\right).$$
The first summation is nonnegative. Each of the terms in the second summation equals $c^2 - c^2=0$. Thus, the whole expression is nonnegative, which completes the proof of the Proposition~\ref{prop}.
\end{proof}

\section{Extending Proposition~\ref{prop}}\label{sec:conjecture}
While the convolution of two general MTP$_2$ densities is not necessarily MTP$_2$ (cf. \cite[pp. 486]{KR}), we saw in Proposition~\ref{prop} that convolving a uniform MTP$_2$ distribution with a scaled standard Gaussian is MTP$_2$. In fact, we conjecture  that convolving {\em any} MTP$_2$ distribution with a scaled standard Gaussian remains MTP$_2$.

\begin{conjecture}\label{conj}
Let $Q$ be an MTP$_2$ distribution. Then, $Q$ satisfies Constraint~A.
\end{conjecture}

While we haven't been  able to prove this conjecture, we know that by Theorem~\ref{thm:main} any convolution of the uniform distribution $U_M$ of a min-max closed set $M$  with  a scaled standard Gaussian is MTP$_2$. Therefore, if we convolve once again with a scaled standard Gaussian, by  associativity of convolution, the resulting density will equal to the original distribution $U_M$ convolved with the convolution of the two scaled standard Gaussians, which is also a scaled standard Gaussian. Thus, we can show Conjecture~\ref{conj} is true for distributions which are obtained as convolutions of $U_M$ and a scaled standard Gaussian.

Moreover, since Conjecture~\ref{conj}  is true for the uniform distribution on a min-max closed set, intuitively, such a distribution could approximate any MTP$_2$ distribution on the vertices of a unit hypercube. A rigorous proof of this statement would yield a proof of Conjecture~\ref{conj}.

\section{Algorithms}\label{sec:algs}
The most intensive task for computing the TPKDE is finding the min-max closure MM$(X)$ of the set of samples $X$.
The cardinality of $MM(X)$ appears to be on the order of (and is certainly upper bounded by) $n^d$, where $n$ is the cardinality of the original set $X$. Therefore computing this set is possibly exponentially hard. Due to the upper bound $n^d$, however, the algorithms are guaranteed to terminate. We designed and implemented several different algorithms for finding MM$(X)$.

The Naive method, Algorithm 1 simply goes through each pair of points in the set and adds their element-wise minima and maxima to $MM(X)$. It then iterates over the points in the current $MM(X)$ and repeats the procedure. It is guaranteed to converge when the min-max closure is found, which is bounded in size by $n^d$. However, its  worst case run-time is $\Omega(n^{2d})$ since we may have to iterate over all $n^d \times n^d$ pairs of points before termination.

\begin{algorithm}[H]
\caption{Naive Algorithm for $MM(X)$}
\begin{algorithmic}[1]
\State $M \leftarrow X$
\State $Run \leftarrow 1$
\While{$Run = 1$}
    \State $oldSize \leftarrow |M|$
    \For{$(i < j \leq oldSize)$} 
        \State add $X_i \wedge X_j$ to $M$
        \State add $X_i \vee X_j$ to $M$
    \EndFor
    \If{$|M| = oldSize$}
         \State $Run \leftarrow 0$
    \EndIf
\EndWhile
\State \Return $M$
\end{algorithmic}
\end{algorithm}

To improve performance over Algorithm 1, we designed Algorithm 2 to utilize the parallel compute capabilities of Nvidia GPUs and the PyCUDA\cite{klockner_pycuda_2012} framework. In describing this algorithm, we take as defined the $MakeGrid$ method (and its inverse $MakeSet$) which takes as input the set $X$ and generates the $d$-dimensional grid representation of all possible $n^d$ points that can be formed by selecting, for $1 < i < d$, the $i$-th coordinate of one of the $n$ points in $X$. We previously mentioned that the min-max closure $MM(X)$ is contained in this finite set of points - as such, this is the worst case result for $MM(X)$, but is much easier to work with on a GPU. The output from $MakeGrid$ is actually further simplified as we can get rid of the numerical value of the points completely and just replace a given coordinate $i$ with its position in the sorted list of $i$-th coordinates of all points in $X$. For example, the coordinate-wise minimum of all points in $X$ can be represented as $(0,0, \dots ,0) \in [n]^d$. This grid is easily represented as an $n$-regular, $d$-dimesnional array (of size $n^d$) where the presence of a point in $X$ is indicated by the value 1 at its representative index in the array which is filled with 0s otherwise. We can add a point to the set represented by this grid by simply setting the desired index to 1, and write $G.SetOne(i)$ to represent setting index $i\in[n]^d$ of grid $G$ to 1. We also take as defined the $GetPoints$ method, which simply returns the $d$-dimensional indices of all the entries in the set which are 1, that is, a list of all the points (represented in $[n]^d$) currently in the set described by the grid $G$.


For Algorithm 2, we structure the computation to be more GPU-friendly than in Algorithm 1. To do this, we simply split the loop over all pairs of points into two nested loops and only compute the first one in a parallel manner. This allows parts of the computation to be done in parallel on a GPU.
\begin{algorithm}[H]
\caption{Optimized Parallel Algorithm for $MM(X)$}
\begin{algorithmic}[1]
\State $G \leftarrow MakeGrid(X)$
\State $P \leftarrow GetPoints(G)$
\State $Run \leftarrow 1$
\While{$Run = 1$}
    \State $oldSize \leftarrow |P|$
    \For{$pt_1$ in $P$} \Comment{This for loop is parallel}
        \For{$pt_2$ in $P$} \Comment{only consider unique pairs}
            \State $G.SetOne(pt_1 \vee pt_2)$
            \State $G.SetOne(pt_1 \wedge pt_2)$
        \EndFor
    \EndFor
    \State $P \leftarrow GetPoints(G)$
        \If{$|P| = oldSize$}
        \State $Run \leftarrow 0$
    \EndIf
\EndWhile
\State \Return $MakeSet(G)$
\end{algorithmic}
\end{algorithm}
The worst-case run-time for both algorithms is $\Omega(n^d)$ since in the worst case we do have to return $n^d$ points. However we see very significant performance improvements using Algorithm 2 over Algorithm 1, and we discuss the actual speedups in the next section.

\section{Experiments}\label{sec:experiments}

\subsection{Min-max closure Algorithm Performance}\label{sec:perf}
We compare the performance of Algorithms 1 and 2 below using randomly generated standard multivariate Gaussian samples of size $n$ and dimensionality $d$ and then using the algorithm to compute the min-max closure the sample set. The relative speed of the algorithms is given in Tables \ref{table:speed2d}-\ref{table:speed4d}.

\begin{table}[H]
\caption{Speedup over Algorithm 1 for $d = 2$}
\centering
\begin{tabular}{c c c c c}
\hline\hline
 & $n=40$ & $n=50$ & $n=60$ & $n=80$\\ [0.5ex]
\hline
Algorithm 2 & $96.33\times$ & $159.49\times$ & $231.96\times$ & $426.04\times$\\
Algorithm 1 & $1.00\times$ & $1.00\times$ & $1.00\times$ & $1.00\times$\\ [1ex]
\hline
\end{tabular}\label{table:speed2d}

\caption{Speedup over Algorithm 1 for $d = 3$}
\centering
\begin{tabular}{c c c c c}
\hline\hline
 & $n=10$ & $n=15$ & $n=20$ & $n=25$\\ [0.5ex]
\hline
Algorithm 2 & $15.52\times$ & $103.48\times$ & $313.72\times$ & $648.08\times$\\
Algorithm 1 & $1.00\times$ & $1.00\times$ & $1.00\times$ & $1.00\times$\\ [1ex]
\hline
\end{tabular}\label{table:speed3d}

\caption{Speedup over Algorithm 1 for $d = 4$}
\centering
\begin{tabular}{c c c c c}
\hline\hline
 & $n=10$ & $n=12$ & $n=15$ & $n=20$\\ [0.5ex]
\hline
Algorithm 2 & $101.31\times$ & $248.82\times$ & $1089.23\times$ & $3089.55\times$\\
Algorithm 1 & $1.00\times$ & $1.00\times$ & $1.00\times$ & $1.00\times$\\ [1ex]
\hline
\end{tabular}\label{table:speed4d}
\end{table}

As we can see, the parallel Algorithm 2 can lead to significant (up to several thousand times) faster computation of the min-max closure of the sample set needed for the TPKDE.

\subsection{Statistical Performance of TPKDE}\label{sec:perf}
In order to evaluate the statistical performance and convergence of our estimator, we first need a method of generating MTP$_2$ densities which KDE or TPKDE can try to estimate. In our experiments, we restrict our attention to the case of Gaussian MTP$_2$ densities, and use the previously mentioned result from \cite{karlinGaussian} that a multivariate Gaussian density is MTP$_2$ if and only if the inverse of its covariance matrix is an {\em M-matrix}. We can therefore computationally generate pseudo-random Gaussian MTP$_2$ densities by using a mean vector that is sampled from a standard Gaussian and generating an M-matrix to use as the inverse covariance. To generate such an inverse covariance matrix, we sample the values on the diagonal from a standard Gaussian, and the off-diagonal entries are sampled from the negative of the absolute value of a standard Gaussian - this is to ensure the off-diagonal entries are non-positive as required for an M-matrix. We repeat this procedure until the resulting covariance matrix we obtain is positive semi-definite and can thus produce a well-formed MTP$_2$ Gaussian density.

Once we have a Gaussian MTP$_2$ density which we denote by $f_0$, we can sample from it and calculate the KDE (say $\hat{f}$) and TPKDE (say $\hat{f}_{TP}$). Throughout, we use Silvermann's Rule\cite{silverman_1986} to calculate the bandwidth for both estimators. We calculate the error of each estimator using a Monte Carlo approach, that is, by sampling $Y_1,\ldots, Y_s$ from the true distribution $f_0$ and calculating the average of
$$
   (f_0(Y_i) - f_E(Y_i))^2
$$
where $f_E$ is the estimated density (either $\hat{f}$ or $\hat{f}_{TP}$). We then take the square root of this value to obtain the estimated root mean square error.
\begin{figure}[H]
    \includegraphics[width=\textwidth]{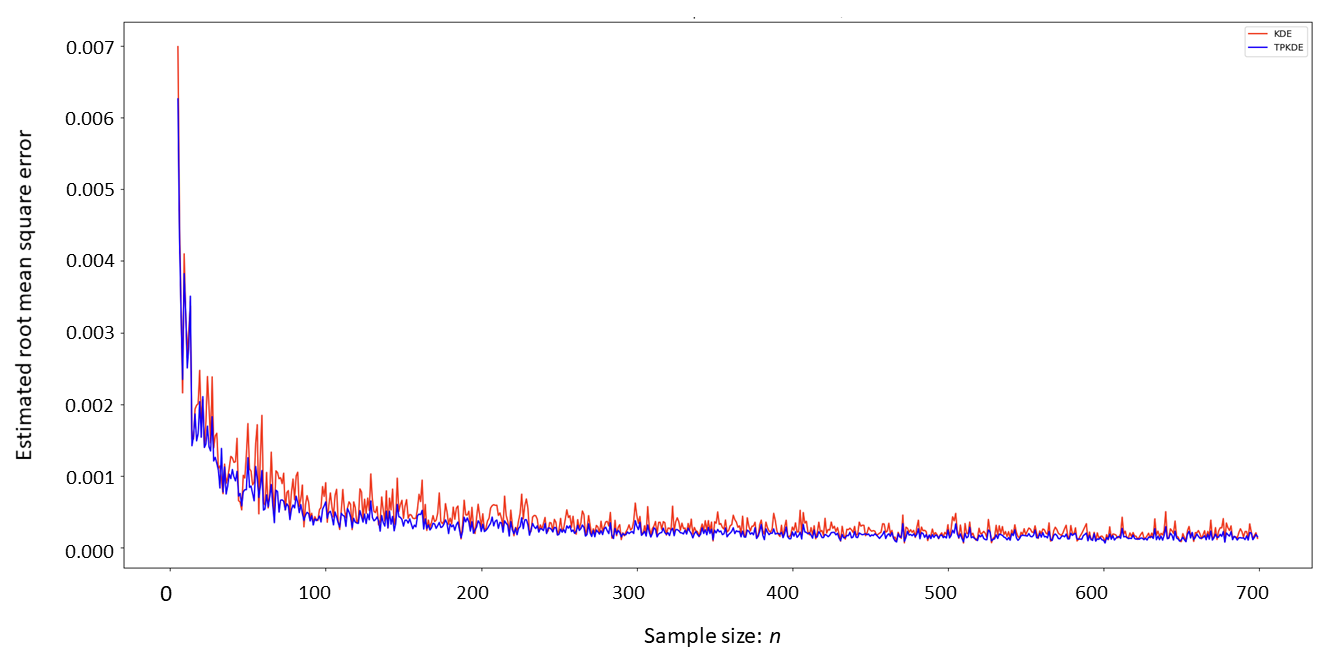}
    \vspace{-0.75cm}
    \caption{Root mean squared error, $d = 2$}
    \label{fig:1}
\end{figure}
 Figures \ref{fig:1}, \ref{fig:2}, and \ref{fig:3} show the comparison of the errors for $d=2$, $d=3$, and $d=4$ respectively.
\begin{figure}[H]
    \includegraphics[width=\textwidth]{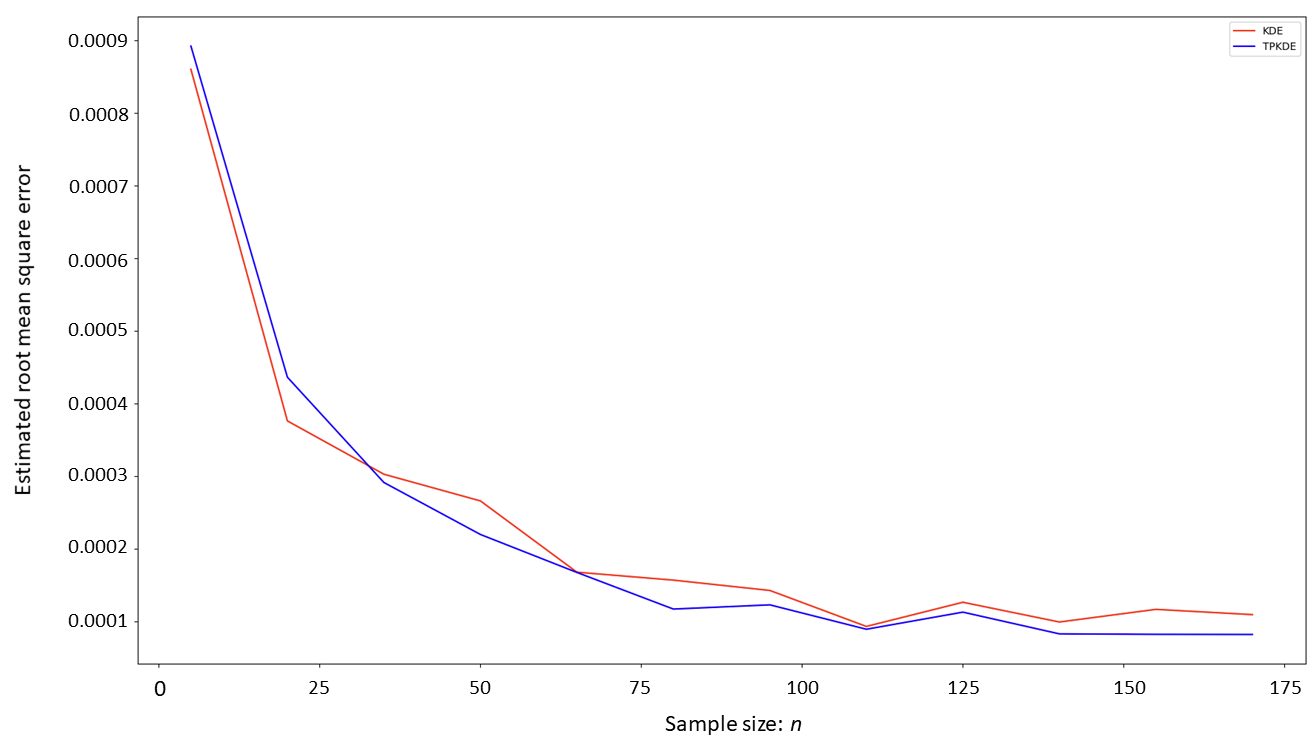}
    \vspace{-0.9cm}
    \caption{Root mean squared error, $d = 3$}
    \label{fig:2}
\end{figure}
\begin{figure}[H]
    \includegraphics[width=\textwidth]{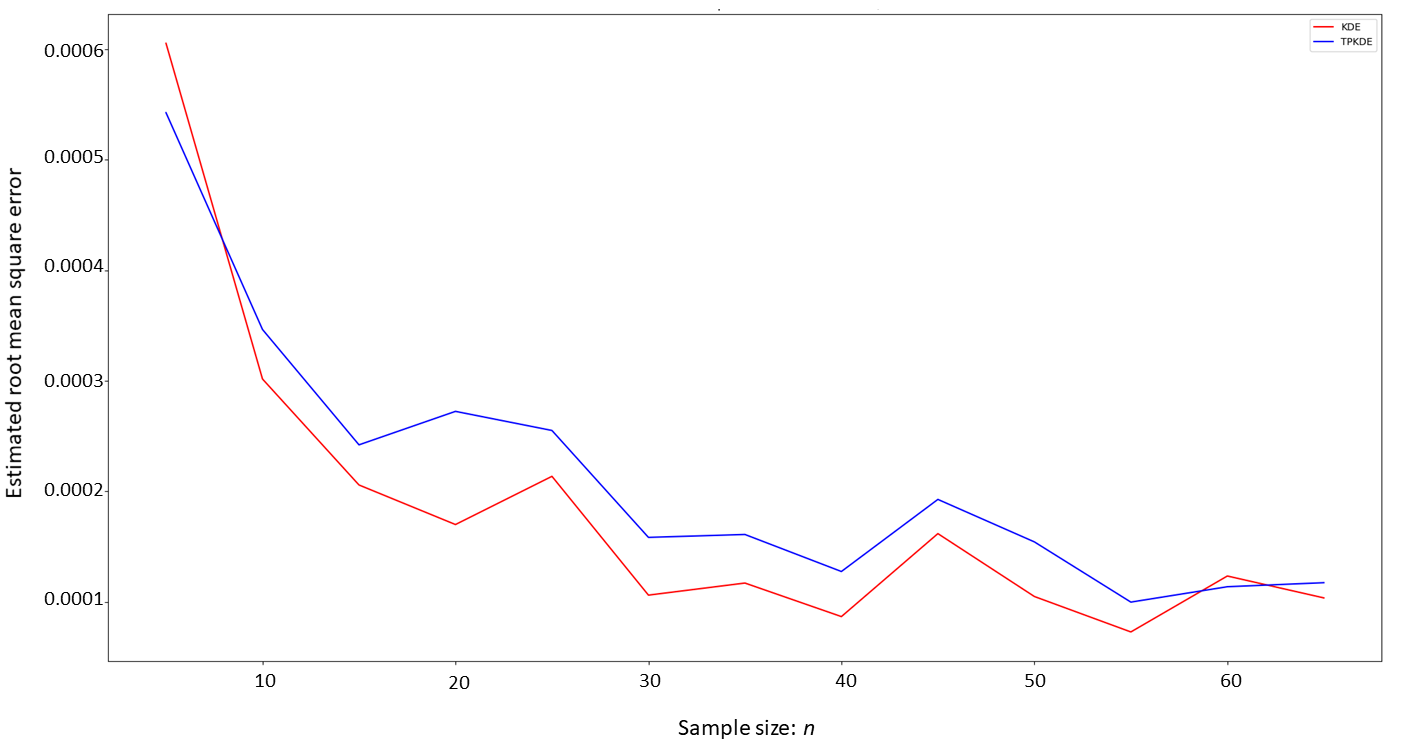}
    \vspace{-0.9cm}
    \caption{Root mean squared error, $d = 4$}
    \label{fig:3}
\end{figure}
As we can see from the plots, it seems that the TPKDE is consistent as the error decreases at a similar rate to the KDE (which is known to be consistent). In fact, in most instances, the TPKDE seems to do better than the standard KDE, with slightly lower errors for a given sample size, $n$. Perhaps, depending on the density $f_0$, as $n$ becomes larger, the TPKDE will perform significantly better than the KDE.

\section{Conclusion}\label{sec:conclusion}
In this paper we considered estimating a density under the constraint that it is MTP$_2$. We proposed augmenting the set of i.i.d. samples and then convolving with a standard Gaussian kernel, thus, yielding our {\em totally positive kernel density estimator}. In particular, because of the nature of the MTP$_2$ condition, instead of considering the original sample of points $X$, we consider the min-max closure, MM$(X)$. Then, our estimator is the convolution of the uniform distribution on MM$(X)$ with the standard Gaussian kernel. One of our main results, Theorem~\ref{thm:main}, shows that this convolution does indeed yield an MTP$_2$ function, and therefore, our estimator is proper. It is not in general true that convolutions of MTP$_2$ functions are MTP$_2$. In Theorem~\ref{thm:conv} we give a \textit{sufficient} condition for when the convolution of a function with a standard Gaussian yields an MTP$_2$ function. This raises an interesting general theoretical question.

\begin{question}
What are necessary conditions on $\alpha:\mathbb R^d\to \mathbb R$ so that the convolution  $\int_{x\in\mathbb R^d}\alpha(\mathbf x)N^h_d(\mathbf x-\mathbf a)d\mathbf x$ is an MTP$_2$ function?
\end{question}


\begin{problem}
Develop the necessary mathematical tools to show consistency of the totally positive kernel density estimator, as well as, of other estimators which rely on a highly dependent set of samples.
\end{problem}

\begin{problem}
Compute the statistical rate of convergence of the TPKDE to the original true density.
\end{problem}

We expect the statistical rate of convergence to be highly dependant on $h$, the choice of bandwidth, which raises the question
\begin{question} What is the optimal choice of bandwidth, $h^*$, of the TPKDE?
\end{question}

We conjecture that due to the large size of MM$(X)$, this statistical rate will not suffer from the curse of dimensionality. The size of MM$(X)$ appears to be on the order of $n^d$, although this is also an interesting mathematical question that needs a proof.

\begin{problem} Compute the expected size of MM$(X)$ when $X$ is drawn i.i.d. from a distribution $f_0$. How do different types of distributions $f_0$ influence this expected size?
\end{problem}

Finally, perhaps we could use a different MTP$_2$ density on MM$(X)$ instead of the uniform one to get better convergence. For this we would have to first prove Conjecture~\ref{conj}.

\begin{problem} 
Are there MTP$_2$ distributions on MM$(X)$ other  than the uniform one that would lead to faster convergence of the TPKDE?
\end{problem}

We are hopeful that the resolution of the above mentioned problems would lead to interesting practical and theoretical advances. 

\section{Acknowledgements}\label{sec:acknowledgements}
Elina Robeva was supported by an NSF postdoctoral fellowship (DMS-170-3821) and an NSERC Discovery Grant (DGECR-2020-00338). Ali Zartash was supported by an Undergraduate Research Opportunity (UROP) at MIT.


\bibliographystyle{alpha}
\bibliography{non_parametric}

\appendix

\section{Proof of Theorem~\ref{thm:conv}}\label{appendix:A}
\begin{proof}[Proof of Theorem~\ref{thm:conv}]
We want to show that $C(\cdot)$, which is the convolution of a scaled standard multivariate Gaussian $N_d^h(\cdot)$ and distribution $Q$ on $\mathbb R^d$ which either has density $\alpha$ on $\mathbb R^d$ or is discrete with weights $\alpha(x) = Q(\{x\})$ is MTP$_2$ when $\alpha(\cdot)$ satisfies Constraint A. We know that $C(\cdot) : \R^d \rightarrow \R^+$ is strictly positive everywhere since it is a convolution involving a Gaussian kernel. The result by Karlin and Rinott in \cite{KR} states a $d$-dimensional strictly positive function is MTP$_2$ if and only if it is MTP$_2$ in 2 coordinates, with the rest held constant (i.e. pairwise MTP$_2$ property implies the global MTP$_2$ property for strictly positive measures). Thus it is sufficient to show that $C(\cdot)$ is MTP$_2$ in the first two coordinates as that would imply pairwise MTP$_2$ since the multivariate Gaussian kernel is symmetric in all coordinates.

Specifically we need that for all $a_1, a_2, b_1, b_2 \in \R$ with $a_1 > a_2$ and $b_1 > b_2$, and some fixed $c_{3}^{d} \in \R^{d-2}$, where use the notation $c_{3}^{d}$ to denote $(c_3, \dots , c_{d})$, the condition
$$
C(a_1, b_1, c_{3}^{d})C(a_2, b_2, c_{3}^{d}) \geq C(a_1, b_2, c_{3}^{d})C(a_2, b_1, c_{3}^{d})
$$
is satisfied. We have described $Q$ as a distribution which either is either density $\alpha$ on $\mathbb R^d$ or is discrete with weights $\alpha(x) = Q(\{x\})$. The latter case applies in our estimator since it takes $\alpha$ to be the uniform distribution over a finite set (specifically the min-max closure of the sample set). The proof is almost identical in both cases, as we can replace integrals with summations throughout the following analysis. In the case where $\alpha$ is a density, we can expand $C(\cdot)$ using the definition of convolution in \eqref{eqn:conv} to get the condition that
{\fontsize{8}{2}\selectfont
\begin{align*}
    &\left( \idotsint\displaylimits_{x_1, \dots ,x_d} \alpha \left(\mathbf{x}\right) N_d^h\left(a_1 - x_1, b_1 - x_2, (c-x)_{3}^d\right)d\mathbf{x} \right) 
    \left( \idotsint\displaylimits_{x_1, \dots ,x_d} \alpha \left(\mathbf{x}\right) N_d^h\left(a_2 - x_1, b_2 - x_2, (c-x)_{3}^d\right)d\mathbf{x} \right) \nonumber \\
    &\geq \left( \idotsint\displaylimits_{x_1, \dots ,x_d} \alpha \left(\mathbf{x}\right) N_d^h\left(a_1 - x_1, b_2 - x_2, (c-x)_{3}^d\right)d\mathbf{x} \right) 
    \left( \idotsint\displaylimits_{x_1, \dots ,x_d} \alpha \left(\mathbf{x}\right) N_d^h\left(a_2 - x_1, b_1 - x_2, (c-x)_{3}^d\right)d\mathbf{x} \right)
\end{align*}
}%
where we use the notation $(x-c)_3^d$ to denote $(x_3 - c_3, \dots, x_d - c_d)$, and $\mathbf{x}$ for $(x_1,\dots,x_n)$. We also drop $dx_1\dots dx_d$  throughout the proof to ease notation and since the domain of integration is clearly stated throughout.\\
We can multiply both sides by $\frac{1}{{2\pi h^d}}$ and use vector norm notation, where we use $\begin{Vmatrix}\cdot \end{Vmatrix}$ to denote the $L_2$ norm. We also use the simple integral $\int$ notation instead of $\idotsint$ for simpler notation since the domain is clearly stated. We get
\begin{align}
    &\left( \int\displaylimits_{\mathbf{x} \in \R^d} \alpha(\mathbf{x}) e^{\frac{-1}{2h}\left(
    \begin{Vmatrix}
    a_1 - x_1\\
    b_1 - x_2\\
    c_3^d - x_3^d
    \end{Vmatrix}
    \right)} \right)
    \left( \int\displaylimits_{\mathbf{x} \in \R^d} \alpha(\mathbf{x}) e^{\frac{-1}{2h}\left(
    \begin{Vmatrix}
    a_2 - x_1\\
    b_2 - x_2\\
    c_3^d - x_3^d
    \end{Vmatrix}
    \right)} \right)\nonumber\\
    & \geq
    \left( \int\displaylimits_{\mathbf{x} \in \R^d} \alpha(\mathbf{x}) e^{\frac{-1}{2h}\left(
    \begin{Vmatrix}
    a_1 - x_1\\
    b_2 - x_2\\
    c_3^d - x_3^d
    \end{Vmatrix}
    \right)} \right)
    \left( \int\displaylimits_{\mathbf{x} \in \R^d} \alpha(\mathbf{x}) e^{\frac{-1}{2h}\left(
    \begin{Vmatrix}
    a_2 - x_1\\
    b_1 - x_2\\
    c_3^d - x_3^d
    \end{Vmatrix}
    \right)} \right).
\end{align}

We denote $c_3^d$ by $\mathbf{c} \in \R^{d-2}$. We now rename $x_1$, $x_2$, and $x_3^d$ to $x_i$, $x_j$, and $\mathbf{m} \in \R^{d-2}$ respectively for the factors on the left and also rename $x_1$, $x_2$, and $x_3^d$ to $x_k$, $x_l$, and $\mathbf{n} \in \R^{d-2}$ respectively for the factors on the right. Splitting the factors in the integrals we get
{\fontsize{9}{10}\selectfont \thinmuskip=0mu\medmuskip=0mu\thickmuskip=0mu
\begin{align*}
    &\left( \int\displaylimits_{\mathbf{m} \in \R^{d-2}}
    e^{\frac{-1}{2h}\begin{Vmatrix}\mathbf{c} - \mathbf{m}\end{Vmatrix}}
    \int\displaylimits_{x_i,x_j \in \R}
    \alpha(x_i,x_j,\mathbf{m}) e^{\frac{-1}{2h}\left(
    \begin{Vmatrix}
    a_1 - x_i\\
    b_1 - x_j
    \end{Vmatrix}
    \right)} \right)
    \left( \int\displaylimits_{\mathbf{n} \in \R^{d-2}}
    e^{\frac{-1}{2h}\begin{Vmatrix}\mathbf{c} - \mathbf{n}\end{Vmatrix}}
    \int\displaylimits_{x_k,x_l \in \R}
    \alpha(x_k,x_l,\mathbf{n}) e^{\frac{-1}{2h}\left(
    \begin{Vmatrix}
    a_2 - x_k\\
    b_2 - x_l
    \end{Vmatrix}
    \right)} \right) \nonumber\\
    & \geq
    \left( \int\displaylimits_{\mathbf{m} \in \R^{d-2}}
    e^{\frac{-1}{2h}\begin{Vmatrix}\mathbf{c} - \mathbf{m}\end{Vmatrix}}
    \int\displaylimits_{x_i,x_j \in \R}
    \alpha(x_i,x_j,\mathbf{m}) e^{\frac{-1}{2h}\left(
    \begin{Vmatrix}
    a_1 - x_i\\
    b_2 - x_j
    \end{Vmatrix}
    \right)} \right)
    \left( \int\displaylimits_{\mathbf{n} \in \R^{d-2}}
    e^{\frac{-1}{2h}\begin{Vmatrix}\mathbf{c} - \mathbf{n}\end{Vmatrix}}
    \int\displaylimits_{x_k,x_l \in \R}
     \alpha(x_k,x_l,\mathbf{n}) e^{\frac{-1}{2h}\left(
    \begin{Vmatrix}
    a_2 - x_k\\
    b_1 - x_l
    \end{Vmatrix}
    \right)} \right).\\
\end{align*}  
}%
Combining the factors into one integral and using the notation $\alpha_{ijm}$ for $\alpha(x_i,x_j,\mathbf{m})$, we get
{
\begin{align}
    \nonumber
    \int\displaylimits_{\mathbf{m,n} \in \R^{d-2}}
    e^{\frac{-1}{2h}\begin{Vmatrix}\mathbf{c} - \mathbf{m}\\\mathbf{c} - \mathbf{n}\end{Vmatrix}}
    &\Bigg[
    \int\displaylimits_{x_i,x_j,x_k,x_l \in \R}
    \alpha_{ijm}\alpha_{kln}
    e^{\frac{-1}{2h}\left(
    \begin{Vmatrix}
    a_1 - x_i\\
    b_1 - x_j
    \end{Vmatrix}
    +
    \begin{Vmatrix}
    a_2 - x_k\\
    b_2 - x_l
    \end{Vmatrix}
    \right)}\\
    &-
    \int\displaylimits_{x_i,x_j,x_k,x_l \in \R}
    \alpha_{ijm}\alpha_{kln}
    e^{\frac{-1}{2h}\left(
    \begin{Vmatrix}
    a_1 - x_i\\
    b_2 - x_j
    \end{Vmatrix}
    +
    \begin{Vmatrix}
    a_2 - x_k\\
    b_1 - x_l
    \end{Vmatrix}
    \right)}
    \Bigg]
    \geq 0.
\end{align}
}%
Let us now examine the inner integrals in (A.2) in more detail by dividing into the regions of integration. We can divide the region of $\int\displaylimits_{x_i,x_j,x_k,x_l \in \R}$ into 7 cases corresponding to each of\\\\
Case 1: $\int\displaylimits_{\substack{x_i > x_k\\ x_j > x_l}}$, 
Case 2: $\int\displaylimits_{\substack{x_i > x_k\\ x_j < x_l}}$, 
Case 3: $\int\displaylimits_{\substack{x_i < x_k\\ x_j > x_l}}$, 
Case 4: $\int\displaylimits_{\substack{x_i < x_k\\ x_j < x_l}}$, 
Case 5: $\int\displaylimits_{\substack{x_i = x_k\\ x_j = x_l}}$ \\\\
Case 6: $\int\displaylimits_{\substack{x_i = x_k\\ x_j \neq x_l}}$, and 
Case 7: $\int\displaylimits_{\substack{x_i\neq x_k\\ x_j = x_l}}$.
\\\\
We will look at the different cases separately. For Case 5 ($\int\displaylimits_{\substack{x_i = x_k\\ x_j = x_l}}$) we have
\begin{align*}
    \int\displaylimits_{\mathbf{m,n} \in \R^{d-2}}
    e^{\frac{-1}{2h}\begin{Vmatrix}\mathbf{c} - \mathbf{m}\\\mathbf{c} - \mathbf{n}\end{Vmatrix}}
    &\Bigg[
    \int\displaylimits_{\substack{x_i = x_k\\ x_j = x_l} }
    \alpha_{ijm}\alpha_{kln}
    e^{\frac{-1}{2h}\left(
    \begin{Vmatrix}
    a_1 - x_i\\
    b_1 - x_j
    \end{Vmatrix}
    +
    \begin{Vmatrix}
    a_2 - x_k\\
    b_2 - x_l
    \end{Vmatrix}
    \right)}\\
    &-
    \int\displaylimits_{\substack{x_i = x_k\\ x_j = x_l} }
    \alpha_{ijm}\alpha_{kln}
    e^{\frac{-1}{2h}\left(
    \begin{Vmatrix}
    a_1 - x_i\\
    b_2 - x_j
    \end{Vmatrix}
    +
    \begin{Vmatrix}
    a_2 - x_k\\
    b_1 - x_l
    \end{Vmatrix}
    \right)}
    \Bigg]
    &= 0
\end{align*}
and this vanishes since the inner terms are the same. For Case 7 ($\int\displaylimits_{\substack{x_i\neq x_k\\ x_j = x_l}}$) we can see that
{\fontsize{10}{11}\selectfont \thinmuskip=0mu\medmuskip=0mu\thickmuskip=0mu
\begin{align*}
    \int\displaylimits_{\mathbf{m,n} \in \R^{d-2}}
    e^{\frac{-1}{2h}\begin{Vmatrix}\mathbf{c} - \mathbf{m}\\\mathbf{c} - \mathbf{n}\end{Vmatrix}}
    \Bigg[
    &\int\displaylimits_{\substack{x_i\neq x_k\\ x_j = x_l} }
    \alpha_{ijm}\alpha_{kln}
    e^{\frac{-1}{2h}\left(
    \begin{Vmatrix}
    a_1 - x_i\\
    b_1 - x_j
    \end{Vmatrix}
    +
    \begin{Vmatrix}
    a_2 - x_k\\
    b_2 - x_l
    \end{Vmatrix}
    \right)}\\
    &-
    \int\displaylimits_{\substack{x_i\neq x_k\\ x_j = x_l} }
    \alpha_{ijm}\alpha_{kln}
    e^{\frac{-1}{2h}\left(
    \begin{Vmatrix}
    a_1 - x_i\\
    b_2 - x_j
    \end{Vmatrix}
    +
    \begin{Vmatrix}
    a_2 - x_k\\
    b_1 - x_l
    \end{Vmatrix}
    \right)}
    \Bigg]
    \\
    =
    \int\displaylimits_{\mathbf{m,n} \in \R^{d-2}}
    e^{\frac{-1}{2h}\begin{Vmatrix}\mathbf{c} - \mathbf{m}\\\mathbf{c} - \mathbf{n}\end{Vmatrix}}
    \Bigg[
    &\int\displaylimits_{\substack{x_i\neq x_k\\ x_j} }
    \alpha_{ijm}\alpha_{kjn}
    e^{\frac{-1}{2h}\left(
    \begin{Vmatrix}
    a_1 - x_i\\
    b_1 - x_j
    \end{Vmatrix}
    +
    \begin{Vmatrix}
    a_2 - x_k\\
    b_2 - x_j
    \end{Vmatrix}
    \right)}\\
    &-
    \int\displaylimits_{\substack{x_i\neq x_k\\ x_j} }
    \alpha_{ijm}\alpha_{kjn}
    e^{\frac{-1}{2h}\left(
    \begin{Vmatrix}
    a_1 - x_i\\
    b_2 - x_j
    \end{Vmatrix}
    +
    \begin{Vmatrix}
    a_2 - x_k\\
    b_1 - x_j
    \end{Vmatrix}
    \right)}
    \Bigg] = 0
\end{align*}
}
And for Case 6 ($\int\displaylimits_{\substack{x_i = x_k\\ x_j \neq x_l}}$) we can simplify the inner terms as
{\fontsize{9}{10}\selectfont \thinmuskip=0mu\medmuskip=0mu\thickmuskip=0mu
\begin{align*}
    &\int\displaylimits_{\substack{x_i = x_k\\ x_j \neq x_l} }
    \alpha_{ijm}\alpha_{kln}
    e^{\frac{-1}{2h}\left(
    \begin{Vmatrix}
    a_1 - x_i\\
    b_1 - x_j
    \end{Vmatrix}
    +
    \begin{Vmatrix}
    a_2 - x_k\\
    b_2 - x_l
    \end{Vmatrix}
    \right)}
    -
    \int\displaylimits_{\substack{x_i = x_k\\ x_j \neq x_l} }
    \alpha_{ijm}\alpha_{kln}
    e^{\frac{-1}{2h}\left(
    \begin{Vmatrix}
    a_1 - x_i\\
    b_2 - x_j
    \end{Vmatrix}
    +
    \begin{Vmatrix}
    a_2 - x_k\\
    b_1 - x_l
    \end{Vmatrix}
    \right)}\\
    &=
    \int\displaylimits_{\substack{x_i\\ x_j \neq x_l} }
    \alpha_{ijm}\alpha_{iln}
    e^{\frac{-1}{2h}\left(
    \begin{Vmatrix}
    a_1 - x_i\\
    b_1 - x_j
    \end{Vmatrix}
    +
    \begin{Vmatrix}
    a_2 - x_i\\
    b_2 - x_l
    \end{Vmatrix}
    \right)}
    -
    \int\displaylimits_{\substack{x_i\\ x_j \neq x_l} }
    \alpha_{ijm}\alpha_{iln}
    e^{\frac{-1}{2h}\left(
    \begin{Vmatrix}
    a_1 - x_i\\
    b_2 - x_j
    \end{Vmatrix}
    +
    \begin{Vmatrix}
    a_2 - x_i\\
    b_1 - x_l
    \end{Vmatrix}
    \right)}\\
    &=
    \int\displaylimits_{\substack{x_i\\ x_j \neq x_l} }
    \alpha_{ijm}\alpha_{iln}
    e^{\frac{-1}{2h}\left(
    \begin{Vmatrix}
    a_1 - x_i\\
    b_1 - x_j
    \end{Vmatrix}
    +
    \begin{Vmatrix}
    a_2 - x_i\\
    b_2 - x_l
    \end{Vmatrix}
    \right)}
    -
    \int\displaylimits_{\substack{x_i\\ x_l \neq x_j} }
    \alpha_{ilm}\alpha_{ijn}
    e^{\frac{-1}{2h}\left(
    \begin{Vmatrix}
    a_1 - x_i\\
    b_2 - x_l
    \end{Vmatrix}
    +
    \begin{Vmatrix}
    a_2 - x_i\\
    b_1 - x_j
    \end{Vmatrix}
    \right)}\\
    &=
    \int\displaylimits_{\substack{x_i\\ x_l \neq x_j} }
    (\alpha_{ijm}\alpha_{iln} - \alpha_{ijn}\alpha_{ilm})
    e^{\frac{-1}{2h}\left(
    \begin{Vmatrix}
    a_1 - x_i\\
    b_2 - x_l
    \end{Vmatrix}
    +
    \begin{Vmatrix}
    a_2 - x_i\\
    b_1 - x_j
    \end{Vmatrix}
    \right)}.
\end{align*}
}
So for this case we get
\begin{align}
    \int\displaylimits_{\mathbf{m,n} \in \R^{d-2}}
    e^{\frac{-1}{2h}\begin{Vmatrix}\mathbf{c} - \mathbf{m}\\\mathbf{c} - \mathbf{n}\end{Vmatrix}}
    \Bigg[
    \int\displaylimits_{\substack{x_i\\ x_l \neq x_j} }
    (\alpha_{ijm}\alpha_{iln} - \alpha_{ijn}\alpha_{ilm})
    e^{\frac{-1}{2h}\left(
    \begin{Vmatrix}
    a_1 - x_i\\
    b_2 - x_l
    \end{Vmatrix}
    +
    \begin{Vmatrix}
    a_2 - x_i\\
    b_1 - x_j
    \end{Vmatrix}
    \right)}
    \Bigg] = 0
\end{align}
This is because the outer integral
$
\int\displaylimits_{\mathbf{m,n} \in \R^{d-2}}
$
over $\mathbf{m,n} \in \R^{d-2}$ is symmetric in $m,n$, and so we can switch $m$ and $n$ in one of the terms. Hence, $\int\displaylimits_{\mathbf{m,n} \in \R^{d-2}} (\alpha_{ijm}\alpha_{iln} - \alpha_{ijn}\alpha_{ilm}) = 0$, and this implies (A.3).
\\\\
Now for the remaining cases i.e. Case 1, 2, 3, and 4, we have
{\fontsize{10}{11}\selectfont \thinmuskip=0mu\medmuskip=0mu\thickmuskip=0mu
\begin{align*}
    \int\displaylimits_{\mathbf{m,n} \in \R^{d-2}}
    e^{\frac{-1}{2h}\begin{Vmatrix}\mathbf{c} - \mathbf{m}\\\mathbf{c} - \mathbf{n}\end{Vmatrix}}
    \Bigg[
    &\int\displaylimits_{\substack{x_i > x_k\\ x_j < x_l}}
    \alpha_{ijm}\alpha_{kln}
    e^{\frac{-1}{2h}\left(
    \begin{Vmatrix}
    a_1 - x_i\\
    b_1 - x_j
    \end{Vmatrix}
    +
    \begin{Vmatrix}
    a_2 - x_k\\
    b_2 - x_l
    \end{Vmatrix}
    \right)}
    -
    \int\displaylimits_{\substack{x_i > x_k\\ x_j < x_l} }
    \alpha_{ijm}\alpha_{kln}
    e^{\frac{-1}{2h}\left(
    \begin{Vmatrix}
    a_1 - x_i\\
    b_2 - x_j
    \end{Vmatrix}
    +
    \begin{Vmatrix}
    a_2 - x_k\\
    b_1 - x_l
    \end{Vmatrix}
    \right)} +\\
    &\int\displaylimits_{\substack{x_i > x_k\\ x_j < x_l}}
    \alpha_{ijm}\alpha_{kln}
    e^{\frac{-1}{2h}\left(
    \begin{Vmatrix}
    a_1 - x_i\\
    b_1 - x_j
    \end{Vmatrix}
    +
    \begin{Vmatrix}
    a_2 - x_k\\
    b_2 - x_l
    \end{Vmatrix}
    \right)}
    -
    \int\displaylimits_{\substack{x_i > x_k\\ x_j > x_l}}
    \alpha_{ijm}\alpha_{kln}
    e^{\frac{-1}{2h}\left(
    \begin{Vmatrix}
    a_1 - x_i\\
    b_2 - x_j
    \end{Vmatrix}
    +
    \begin{Vmatrix}
    a_2 - x_k\\
    b_1 - x_l
    \end{Vmatrix}
    \right)} + \\
    &\int\displaylimits_{\substack{x_i < x_k\\ x_j > x_l}}
    \alpha_{ijm}\alpha_{kln}
    e^{\frac{-1}{2h}\left(
    \begin{Vmatrix}
    a_1 - x_i\\
    b_1 - x_j
    \end{Vmatrix}
    +
    \begin{Vmatrix}
    a_2 - x_k\\
    b_2 - x_l
    \end{Vmatrix}
    \right)}
    -
    \int\displaylimits_{\substack{x_i < x_k\\ x_j < x_l} }
    \alpha_{ijm}\alpha_{kln}
    e^{\frac{-1}{2h}\left(
    \begin{Vmatrix}
    a_1 - x_i\\
    b_2 - x_j
    \end{Vmatrix}
    +
    \begin{Vmatrix}
    a_2 - x_k\\
    b_1 - x_l
    \end{Vmatrix}
    \right)} +\\
    &\int\displaylimits_{\substack{x_i < x_k\\ x_j < x_l} }
    \alpha_{ijm}\alpha_{kln}
    e^{\frac{-1}{2h}\left(
    \begin{Vmatrix}
    a_1 - x_i\\
    b_1 - x_j
    \end{Vmatrix}
    +
    \begin{Vmatrix}
    a_2 - x_k\\
    b_2 - x_l
    \end{Vmatrix}
    \right)}
    -
    \int\displaylimits_{\substack{x_i < x_k\\ x_j > x_l} }
    \alpha_{ijm}\alpha_{kln}
    e^{\frac{-1}{2h}\left(
    \begin{Vmatrix}
    a_1 - x_i\\
    b_2 - x_j
    \end{Vmatrix}
    +
    \begin{Vmatrix}
    a_2 - x_k\\
    b_1 - x_l
    \end{Vmatrix}
    \right)}
    \Bigg].
\end{align*}
}
Now with renaming indices and manipulation of terms, we get
{\fontsize{10}{11}\selectfont \thinmuskip=0mu\medmuskip=0mu\thickmuskip=0mu
\begin{align*}
    \int\displaylimits_{\mathbf{m,n} \in \R^{d-2}}
    e^{\frac{-1}{2h}\begin{Vmatrix}\mathbf{c} - \mathbf{m}\\\mathbf{c} - \mathbf{n}\end{Vmatrix}}
    \Bigg[
    &\int\displaylimits_{\substack{x_i > x_k\\ x_j > x_l}}
    \alpha_{ijm}\alpha_{kln}
    e^{\frac{-1}{2h}\left(
    \begin{Vmatrix}
    a_1 - x_i\\
    b_1 - x_j
    \end{Vmatrix}
    +
    \begin{Vmatrix}
    a_2 - x_k\\
    b_2 - x_l
    \end{Vmatrix}
    \right)}
    -
    \int\displaylimits_{\substack{x_i > x_k\\ x_j > x_l} }
    \alpha_{ilm}\alpha_{kjn}
    e^{\frac{-1}{2h}\left(
    \begin{Vmatrix}
    a_1 - x_i\\
    b_2 - x_l
    \end{Vmatrix}
    +
    \begin{Vmatrix}
    a_2 - x_k\\
    b_1 - x_j
    \end{Vmatrix}
    \right)} +\\
    &\int\displaylimits_{\substack{x_i > x_k\\ x_j > x_l}}
    \alpha_{ilm}\alpha_{kjn}
    e^{\frac{-1}{2h}\left(
    \begin{Vmatrix}
    a_1 - x_i\\
    b_1 - x_l
    \end{Vmatrix}
    +
    \begin{Vmatrix}
    a_2 - x_k\\
    b_2 - x_j
    \end{Vmatrix}
    \right)}
    -
    \int\displaylimits_{\substack{x_i > x_k\\ x_j > x_l}}
    \alpha_{ijm}\alpha_{kln}
    e^{\frac{-1}{2h}\left(
    \begin{Vmatrix}
    a_1 - x_i\\
    b_2 - x_j
    \end{Vmatrix}
    +
    \begin{Vmatrix}
    a_2 - x_k\\
    b_1 - x_l
    \end{Vmatrix}
    \right)} + \\
    &\int\displaylimits_{\substack{x_i > x_k\\ x_j > x_l}}
    \alpha_{kjm}\alpha_{iln}
    e^{\frac{-1}{2h}\left(
    \begin{Vmatrix}
    a_1 - x_k\\
    b_1 - x_j
    \end{Vmatrix}
    +
    \begin{Vmatrix}
    a_2 - x_i\\
    b_2 - x_l
    \end{Vmatrix}
    \right)}
    -
    \int\displaylimits_{\substack{x_i > x_k\\ x_j > x_l} }
    \alpha_{klm}\alpha_{ijn}
    e^{\frac{-1}{2h}\left(
    \begin{Vmatrix}
    a_1 - x_k\\
    b_2 - x_l
    \end{Vmatrix}
    +
    \begin{Vmatrix}
    a_2 - x_i\\
    b_1 - x_j
    \end{Vmatrix}
    \right)} +\\
    &\int\displaylimits_{\substack{x_i > x_k\\ x_j > x_l} }
    \alpha_{klm}\alpha_{ijn}
    e^{\frac{-1}{2h}\left(
    \begin{Vmatrix}
    a_1 - x_k\\
    b_1 - x_l
    \end{Vmatrix}
    +
    \begin{Vmatrix}
    a_2 - x_i\\
    b_2 - x_j
    \end{Vmatrix}
    \right)}
    -
    \int\displaylimits_{\substack{x_i > x_k\\ x_j > x_l} }
    \alpha_{kjm}\alpha_{iln}
    e^{\frac{-1}{2h}\left(
    \begin{Vmatrix}
    a_1 - x_k\\
    b_2 - x_j
    \end{Vmatrix}
    +
    \begin{Vmatrix}
    a_2 - x_i\\
    b_1 - x_l
    \end{Vmatrix}
    \right)}
    \Bigg].
\end{align*}
}
We now collect terms and write the whole expression as
\begin{align*}
\int\displaylimits_{\mathbf{m,n} \in \R^{d-2}}
    e^{\frac{-1}{2h}\begin{Vmatrix}\mathbf{c} - \mathbf{m}\\\mathbf{c} - \mathbf{n}\end{Vmatrix}}
    \int\displaylimits_{\substack{x_i > x_k\\ x_j > x_l}}\Bigg[
    &\alpha_{ijm}\alpha_{kln}
    \left(
    e^{\frac{-1}{2h}\left(
    \begin{Vmatrix}
    a_1 - x_i\\
    b_1 - x_j
    \end{Vmatrix}
    +
    \begin{Vmatrix}
    a_2 - x_k\\
    b_2 - x_l
    \end{Vmatrix}
    \right)}
    -
    e^{\frac{-1}{2h}\left(
    \begin{Vmatrix}
    a_1 - x_i\\
    b_2 - x_j
    \end{Vmatrix}
    +
    \begin{Vmatrix}
    a_2 - x_k\\
    b_1 - x_l
    \end{Vmatrix}
    \right)}
    \right)
    \\
    &-
    \alpha_{ilm}\alpha_{kjn}
    \left(
    e^{\frac{-1}{2h}\left(
    \begin{Vmatrix}
    a_1 - x_i\\
    b_2 - x_l
    \end{Vmatrix}
    +
    \begin{Vmatrix}
    a_2 - x_k\\
    b_1 - x_j
    \end{Vmatrix}
    \right)}
    -
    e^{\frac{-1}{2h}\left(
    \begin{Vmatrix}
    a_1 - x_i\\
    b_1 - x_l
    \end{Vmatrix}
    +
    \begin{Vmatrix}
    a_2 - x_k\\
    b_2 - x_j
    \end{Vmatrix}
    \right)}
    \right)
    \\
    &+
    \alpha_{ijn}\alpha_{klm}
    \left(
    e^{\frac{-1}{2h}\left(
    \begin{Vmatrix}
    a_1 - x_k\\
    b_1 - x_l
    \end{Vmatrix}
    +
    \begin{Vmatrix}
    a_2 - x_i\\
    b_2 - x_j
    \end{Vmatrix}
    \right)}
    -
    e^{\frac{-1}{2h}\left(
    \begin{Vmatrix}
    a_1 - x_k\\
    b_2 - x_l
    \end{Vmatrix}
    +
    \begin{Vmatrix}
    a_2 - x_i\\
    b_1 - x_j
    \end{Vmatrix}
    \right)}
    \right)
    \\
    &-
    \alpha_{kjm}\alpha_{iln}
    \left(
    e^{\frac{-1}{2h}\left(
    \begin{Vmatrix}
    a_1 - x_k\\
    b_2 - x_j
    \end{Vmatrix}
    +
    \begin{Vmatrix}
    a_2 - x_i\\
    b_1 - x_l
    \end{Vmatrix}
    \right)}
    -
    e^{\frac{-1}{2h}\left(
    \begin{Vmatrix}
    a_1 - x_k\\
    b_1 - x_j
    \end{Vmatrix}
    +
    \begin{Vmatrix}
    a_2 - x_i\\
    b_2 - x_l
    \end{Vmatrix}
    \right)}
    \right)
    \Bigg].
\end{align*}
Now here the integral in $m,n$ is symmetric and we can switch the names of $m$ and $n$ in the last two terms to get
\begin{align*}
    \int\displaylimits_{\mathbf{m,n} \in \R^{d-2}}
    e^{\frac{-1}{2h}\begin{Vmatrix}\mathbf{c} - \mathbf{m}\\\mathbf{c} - \mathbf{n}\end{Vmatrix}}
    \int\displaylimits_{\substack{x_i > x_k\\ x_j > x_l}}\Bigg[
    \alpha_{ijm}\alpha_{kln}
    \Bigg(
    &e^{\frac{-1}{2h}\left(
    \begin{Vmatrix}
    a_1 - x_i\\
    b_1 - x_j
    \end{Vmatrix}
    +
    \begin{Vmatrix}
    a_2 - x_k\\
    b_2 - x_l
    \end{Vmatrix}
    \right)}
    -
    e^{\frac{-1}{2h}\left(
    \begin{Vmatrix}
    a_1 - x_i\\
    b_2 - x_j
    \end{Vmatrix}
    +
    \begin{Vmatrix}
    a_2 - x_k\\
    b_1 - x_l
    \end{Vmatrix}
    \right)}\\
    &+
    e^{\frac{-1}{2h}\left(
    \begin{Vmatrix}
    a_1 - x_k\\
    b_1 - x_l
    \end{Vmatrix}
    +
    \begin{Vmatrix}
    a_2 - x_i\\
    b_2 - x_j
    \end{Vmatrix}
    \right)}
    -
    e^{\frac{-1}{2h}\left(
    \begin{Vmatrix}
    a_1 - x_k\\
    b_2 - x_l
    \end{Vmatrix}
    +
    \begin{Vmatrix}
    a_2 - x_i\\
    b_1 - x_j
    \end{Vmatrix}
    \right)}
    \Bigg)
    \\
    -
    \alpha_{ilm}\alpha_{kjn}&
    \Bigg(
    e^{\frac{-1}{2h}\left(
    \begin{Vmatrix}
    a_1 - x_i\\
    b_2 - x_l
    \end{Vmatrix}
    +
    \begin{Vmatrix}
    a_2 - x_k\\
    b_1 - x_j
    \end{Vmatrix}
    \right)}
    -
    e^{\frac{-1}{2h}\left(
    \begin{Vmatrix}
    a_1 - x_i\\
    b_1 - x_l
    \end{Vmatrix}
    +
    \begin{Vmatrix}
    a_2 - x_k\\
    b_2 - x_j
    \end{Vmatrix}
    \right)}
    \\
    &+
    e^{\frac{-1}{2h}\left(
    \begin{Vmatrix}
    a_1 - x_k\\
    b_2 - x_j
    \end{Vmatrix}
    +
    \begin{Vmatrix}
    a_2 - x_i\\
    b_1 - x_l
    \end{Vmatrix}
    \right)}
    -
    e^{\frac{-1}{2h}\left(
    \begin{Vmatrix}
    a_1 - x_k\\
    b_1 - x_j
    \end{Vmatrix}
    +
    \begin{Vmatrix}
    a_2 - x_i\\
    b_2 - x_l
    \end{Vmatrix}
    \right)}
    \Bigg)
    \Bigg].
\end{align*}

We can switch order of integrals to get
\begin{align}\label{eqn:appendix_final_cond}
    \nonumber \int\displaylimits_{\substack{x_i > x_k\\ x_j > x_l}}
    \Bigg(
    &e^{\frac{-1}{2h}\left(
    \begin{Vmatrix}
    a_1 - x_i\\
    b_1 - x_j
    \end{Vmatrix}
    +
    \begin{Vmatrix}
    a_2 - x_k\\
    b_2 - x_l
    \end{Vmatrix}
    \right)}
    -
    e^{\frac{-1}{2h}\left(
    \begin{Vmatrix}
    a_1 - x_i\\
    b_2 - x_j
    \end{Vmatrix}
    +
    \begin{Vmatrix}
    a_2 - x_k\\
    b_1 - x_l
    \end{Vmatrix}
    \right)}\\
    &+
    e^{\frac{-1}{2h}\left(
    \begin{Vmatrix}
    a_1 - x_k\\
    b_1 - x_l
    \end{Vmatrix}
    +
    \begin{Vmatrix}
    a_2 - x_i\\
    b_2 - x_j
    \end{Vmatrix}
    \right)}
    -
    e^{\frac{-1}{2h}\left(
    \begin{Vmatrix}
    a_1 - x_k\\
    b_2 - x_l
    \end{Vmatrix}
    +
    \begin{Vmatrix}
    a_2 - x_i\\
    b_1 - x_j
    \end{Vmatrix}
    \right)}
    \Bigg)
    \int\displaylimits_{\mathbf{m,n} \in \R^{d-2}}
    (\alpha_{ijm}\alpha_{kln}-\alpha_{ilm}\alpha_{kjn}).
\end{align}
Now the inner integral in \eqref{eqn:appendix_final_cond},
\begin{align}
\int\displaylimits_{\mathbf{m,n} \in \R^{d-2}}
(\alpha_{ijm}\alpha_{kln}-\alpha_{ilm}\alpha_{kjn}),
\end{align}
is non-negative due to Constraint A. To see this let $x_{11} = x_i$, $x_{21} = x_{j}$, $x_{10} = x_k$, and $x_{20} = x_l$. And let $(x_{31},...,x_{d1}) = \mathbf{m} \vee \mathbf{n}$; similarly let $(x_{30},...,x_{d0}) = \mathbf{m} \wedge \mathbf{n}$. So we have $x_{10} < x_{11}, x_{20} < x_{21}, \ldots, x_{d0} < x_{d1}$. Now we can write (A.5) as \footnote{This is done by splitting the integral into the $2^{d-2}$ orderings of the coordinates of $\mathbf{m}$ and $\mathbf{n}$ and renaming so that our sum/integral is over only one orthant.}

\begin{align*}
\int\displaylimits_{ \substack{x_{30} < x_{31}\\ \ldots\\ x_{d0} < x_{d1}} }
\sum_{a\in\{0,1\}^{d-2}} \left(\alpha(x_{(11a)})\alpha(x_{(00\overline a)}) - \alpha(x_{(10a)}) \alpha(x_{(01\overline a)})\right),
\end{align*}
which can easily be seen as true due to Constraint A. Note that if we considered the above analysis for any other pair of coordinates (above we only considered the first two coordinates), we will use that $\alpha_{ijk} = \alpha(m_1,\dots,x_i,\dots,x_j,\dots,m_d)$, where $x_i$ and $x_j$ could be any two of the $d$ coordinates. Thus in \textit{Constraint A} we require that 
$$\sum_{a\in\{0,1\}^{d-2}} \left(\alpha(x_{\pi(11a)})\alpha(x_{\pi(00\overline a)}) - \alpha(x_{\pi(10a)}) \alpha(x_{\pi(01\overline a)})\right) \geq 0$$
for any \textit{permutation} $\pi$ of $\{1,\dots,d\}$.

Now we have shown the inner summation in \eqref{eqn:appendix_final_cond} is non-negative. We can further say that its factor in \eqref{eqn:appendix_final_cond} is also non-negative due to Lemma~\ref{lem:exppos}. Thus the whole summation is positive, proving Theorem~\ref{thm:conv}.
 
\end{proof}

\begin{lemma}\label{lem:exppos}
For any $a_1,b_1,a_2,b_2 \in \R$ such that $a_1 > a_2, b_1 > b_2$, and any $x_i,x_j,x_k,x_l \in \R$ such that $x_i > x_k$, $x_j > x_l$, we have
\begin{align*}
    &e^{\frac{-1}{2h}\left(
    \begin{Vmatrix}
    a_1 - x_i\\
    b_1 - x_j
    \end{Vmatrix}
    +
    \begin{Vmatrix}
    a_2 - x_k\\
    b_2 - x_l
    \end{Vmatrix}
    \right)}
    -
    e^{\frac{-1}{2h}\left(
    \begin{Vmatrix}
    a_1 - x_i\\
    b_2 - x_j
    \end{Vmatrix}
    +
    \begin{Vmatrix}
    a_2 - x_k\\
    b_1 - x_l
    \end{Vmatrix}
    \right)}\\
    + \ \ \
    &e^{\frac{-1}{2h}\left(
    \begin{Vmatrix}
    a_1 - x_k\\
    b_1 - x_l
    \end{Vmatrix}
    +
    \begin{Vmatrix}
    a_2 - x_i\\
    b_2 - x_j
    \end{Vmatrix}
    \right)}
    -
    e^{\frac{-1}{2h}\left(
    \begin{Vmatrix}
    a_1 - x_k\\
    b_2 - x_l
    \end{Vmatrix}
    +
    \begin{Vmatrix}
    a_2 - x_i\\
    b_1 - x_j
    \end{Vmatrix}
    \right)}
    \ \ \ \geq 0
\end{align*}
where $\begin{Vmatrix}\cdot \end{Vmatrix}$ is used denote the $L_2$ norm.
\end{lemma}
\begin{proof}
We can write the expression as 
\begin{align*}
    e^{\frac{-1}{2h}(a_1^2+a_2^2+b_1^2+b_2^2+x_i^2+x_j^2+x_k^2+x_l^2)}
    \big( &e^{\frac{1}{h}(a_1x_i + b_1x_j + a_2x_k + b_2x_l)}\\
        - &e^{\frac{1}{h}(a_1x_i + b_2x_j + a_2x_k + b_1x_l)}\\
        + &e^{\frac{1}{h}(a_1x_k + b_1x_l + a_2x_i + b_2x_j)}\\
        - &e^{\frac{1}{h}(a_1x_k + b_1x_j + a_2x_i + b_2x_l)}
    \big)
\end{align*}
The common factor $e^{\frac{-1}{2h}(a_1^2+a_2^2+b_1^2+b_2^2+x_i^2+x_j^2+x_k^2+x_l^2)} > 0$ so we only need to show that the sum of the inner terms are non-negative to show that the whole expression is negative. So looking at the inner expression, we have
\begin{align*}
    &\big( \ e^{\frac{1}{h}(a_1x_i + b_1x_j + a_2x_k + b_2x_l)}
        - e^{\frac{1}{h}(a_1x_i + b_2x_j + a_2x_k + b_1x_l)}
        + \ e^{\frac{1}{h}(a_1x_k + b_1x_j + a_2x_i + b_2x_j)}
        - e^{\frac{1}{h}(a_1x_k + b_1x_j + a_2x_i + b_2x_l)} \ \big)\\
    &= e^{\frac{1}{h}(a_1x_i + a_2x_k)}(
    e^{\frac{1}{h}(b_1x_j + b_2x_l)} -
    e^{\frac{1}{h}(b_2x_j + b_1x_l)})
    + e^{\frac{1}{h}(a_1x_k + a_2x_i)}(
    e^{\frac{1}{h}(b_1x_l + b_2x_j)}
    - e^{\frac{1}{h}(b_1x_j + b_2x_l)})\\
    &= (e^{\frac{1}{h}(b_1x_j + b_2x_l)} - e^{\frac{1}{h}(b_2x_j + b_1x_l)})
       (e^{\frac{1}{h}(a_1x_i + a_2x_k)} - e^{\frac{1}{h}(a_1x_k + a_2x_i)})\\
    &\geq 0.
\end{align*}
The inequality at the end follows from the fact that both factors in the expression are non-negative. This is because in both factors, the exponent of the positive term is larger than the exponent of the negative term. Precisely
$$
a_1x_i + a_2x_k - a_1x_k - a_2x_i = (a_1-a_2)(x_i-x_k) \geq 0,
$$
and
$$
b_1x_j + b_2x_lk - b_2x_j - b_1x_l = (b_1-b_2)(x_j-x_l) \geq 0,
$$
where the inequalities here are due to our predicate that $a_1 > a_2, b_1 > b_2$ and $x_i > x_k$, $x_j > x_l$.

\end{proof}

\section{Discussion of Counterexample~\ref{counter:2d_mtp2}}\label{appendix:B}
The Totally Positive Kernel Density estimator (TPKDE) can be interpreted as the convolution of the uniform distribution over a min-max closed set with a scaled \emph{standard} multivariate Gaussian kernel, that is, with covariance matrix $h\Sigma$, where $h$ is scalar, and $\Sigma = I$ is the identity matrix. We show this in (\ref{eqn:TPKDEconv}). We also show in Theorem \ref{thm:main} (using Theorem \ref{thm:conv} and Proposition \ref{prop}) that this convolution yields an MTP$_2$ density. A natural question to ask is whether we can relax the condition on the Gaussian kernel to also include Gaussian kernels with general covariance matrices, that is, Gaussian kernels with covariance matrices that are not the product of a scalar and the identity matrix. Counterexample \ref{counter:2d_mtp2} shows that, in fact, such a relaxation is not possible. Since we use an M-matrix in Counterexample \ref{counter:2d_mtp2}, it also shows that we cannot even relax the condition to {\em MTP$_2$ Gaussians}, that is, to Gaussians with M-matrix covariance matrices, which is a tighter and less general class of Gaussians.

In order to derive this counterexample, we can look at \ref{appendix:A} for the proof of Theorem \ref{thm:conv}, for the two-dimensional Gaussian cases (since Counterexample \ref{counter:2d_mtp2} is for the two-dimensional case). The proof would proceed exactly the same as in \ref{appendix:A} with the $L_2$ norm being replaced by $\begin{Vmatrix}\cdot \end{Vmatrix}_{\Sigma}$ where for any vector $\mathbf{x}$, we have $\begin{Vmatrix} \mathbf{x} \end{Vmatrix}_{\Sigma} = \mathbf{x}^T\Sigma^{-1}\mathbf{x}$. For the purpose of deriving a counterexample, we ignore equality cases when decomposing the integral into regions (i.e. we ignore Cases 5, 6, and 7 - they do not vanish as in \ref{appendix:A}). We end up with a version of \eqref{eqn:appendix_final_cond} which says
\begin{align}\label{eqn:appendix_cond_counter}
    \nonumber \int\displaylimits_{\substack{x_i > x_k\\ x_j > x_l}}
    \Bigg(
    &e^{\frac{-1}{2}\left(
    \begin{Vmatrix}
    a_1 - x_i\\
    b_1 - x_j
    \end{Vmatrix}_\Sigma
    +
    \begin{Vmatrix}
    a_2 - x_k\\
    b_2 - x_l
    \end{Vmatrix}_\Sigma
    \right)}
    -
    e^{\frac{-1}{2}\left(
    \begin{Vmatrix}
    a_1 - x_i\\
    b_2 - x_j
    \end{Vmatrix}_\Sigma
    +
    \begin{Vmatrix}
    a_2 - x_k\\
    b_1 - x_l
    \end{Vmatrix}_\Sigma
    \right)}\\
    &+
    e^{\frac{-1}{2}\left(
    \begin{Vmatrix}
    a_1 - x_k\\
    b_1 - x_l
    \end{Vmatrix}_\Sigma
    +
    \begin{Vmatrix}
    a_2 - x_i\\
    b_2 - x_j
    \end{Vmatrix}_\Sigma
    \right)}
    -
    e^{\frac{-1}{2}\left(
    \begin{Vmatrix}
    a_1 - x_k\\
    b_2 - x_l
    \end{Vmatrix}_\Sigma
    +
    \begin{Vmatrix}
    a_2 - x_i\\
    b_1 - x_j
    \end{Vmatrix}_\Sigma
    \right)}
    \Bigg)
    (\alpha_{ij}\alpha_{kl}-\alpha_{il}\alpha_{kj}).
\end{align}
Now in \eqref{eqn:appendix_cond_counter}, the factor
$(\alpha_{ij}\alpha_{kl}-\alpha_{il}\alpha_{kj})$ is non-negative due to Constraint A (which is the same as the MTP$_2$ condition in 2-dimensions as shown in Remark~\ref{remark:ConsA}). For the whole expression to be non-negative, it would be sufficient for
\begin{align}
\nonumber
&e^{\frac{-1}{2}\left(
    \begin{Vmatrix}
    a_1 - x_i\\
    b_1 - x_j
    \end{Vmatrix}_\Sigma
    +
    \begin{Vmatrix}
    a_2 - x_k\\
    b_2 - x_l
    \end{Vmatrix}_\Sigma
    \right)}
    -
    e^{\frac{-1}{2}\left(
    \begin{Vmatrix}
    a_1 - x_i\\
    b_2 - x_j
    \end{Vmatrix}_\Sigma
    +
    \begin{Vmatrix}
    a_2 - x_k\\
    b_1 - x_l
    \end{Vmatrix}_\Sigma
    \right)}\\
    &+
    e^{\frac{-1}{2}\left(
    \begin{Vmatrix}
    a_1 - x_k\\
    b_1 - x_l
    \end{Vmatrix}_\Sigma
    +
    \begin{Vmatrix}
    a_2 - x_i\\
    b_2 - x_j
    \end{Vmatrix}_\Sigma
    \right)}
    -
    e^{\frac{-1}{2}\left(
    \begin{Vmatrix}
    a_1 - x_k\\
    b_2 - x_l
    \end{Vmatrix}_\Sigma
    +
    \begin{Vmatrix}
    a_2 - x_i\\
    b_1 - x_j
    \end{Vmatrix}_\Sigma
    \right)}
\end{align}
to be non-negative for the whole region of the integral. In \ref{appendix:A} this ends up being the case due to \ref{lem:exppos}. However, here this is not necessarily true. For covariance matrix $\Sigma$ such that
$$
\Sigma^{-1} = \begin{bmatrix}
a&c\\
c&d\\
\end{bmatrix}
$$
with $a > 0, d > 0, c < 0 \in \mathbb{R}$ chosen such that $\Sigma^{-1}$ is a positive-definite M-matrix, we can compare the coefficients of the exponents and get the sufficient condition that 
$$
(a(x_i-x_k)+c(x_j-x_l))(d(x_j-x_l)+c(x_i-x_k)) \geq 0.
$$
This is not satisfied when $(x_i-x_k) > (x_j-x_l)$ and $a > |c| > d$. We can now try out different values of $\Sigma$ and the sample set that violate the sufficient condition. Indeed the set MM$(X) = \{(2,0),(0,1),(2,1),(0,0)\}$ has $(x_i-x_k) > (x_j-x_l)$ for the points $(2,1)$ and $(0,0)$ (i.e. $x_i = 2, x_j = 1, x_k=0,x_j=0$. And
$$
 \Sigma^{-1}= \begin{pmatrix} 5 & -2 \\
                             -2 & 1  
                             \end{pmatrix}
$$
has $5 = a > 2 =|c| > d = 1$. This is exactly the inverse covariance matrix and sample set used in Counterexample \ref{counter:2d_mtp2}.

\end{document}